\documentclass{amsart}
\usepackage{times,amsfonts,amssymb,
amsrefs,mathrsfs,tikz, amsthm}%
\allowdisplaybreaks[2]

\newcommand\Z{\mathbb{Z}}

\newcommand{\IC}{{\mathbb C}}
\newcommand{\IN}{{\mathbb N}}
\newcommand{\IR}{{\mathbb R}}
\newcommand{\IZ}{{\mathbb Z}}
\newcommand{\cH}{{\mathcal H}}
\newcommand{\cK}{{\mathcal K}}

\newcommand{\cS}{{\mathcal S}}
\newcommand{\cU}{U}
\newcommand{\acts}{\curvearrowright}
\newcommand{\ve}{\varepsilon}
\DeclareMathOperator{\IE}{{\mathbb E}}
\DeclareMathOperator{\IP}{{\mathbb P}}
\DeclareMathOperator{\supp}{supp}
\DeclareMathOperator{\cspan}{\overline{span}}
\DeclareMathOperator{\Tr}{Tr}
\DeclareMathOperator{\Ad}{Ad}
\newcommand{\ip}[1]{\mathopen{\langle}#1\mathclose{\rangle}}
\newcommand{\ie}[1]{\operatorname{\mathbb E}\bigl[ #1 \bigr]}

\newtheorem{lemma}{Lemma}[section]

\newtheorem{proposition}[lemma]{Proposition}
\newtheorem{theorem}[lemma]{Theorem}
\newtheorem{corollary}[lemma]{Corollary}
\newtheorem{example}[lemma]{Example}

\newtheorem{rem}[lemma]{Remark}

\newtheorem{maintheorem}{Theorem}
\newtheorem*{corollaryint}{Corollary}

\theoremstyle{definition}
\newtheorem{remark}[lemma]{Remark}
\newtheorem{definition}[lemma]{Definition}

\title{Finite-Dimensional Representations constructed from Random Walks}

\author{Anna Erschler}

\address{A.E.: C.N.R.S., D\'epartement de math\'ematiques et applications, \'Ecole Normale Sup\'erieure, PSL research institute, 45, Rue d'Ulm, 75005, Paris, France}
\email{anna.erschler@ens.fr}

\author{Narutaka Ozawa}

\address{N.O.: Research Institute of Mathematical Sciences, Kyoto University,
Kitashirakawa-Oiwake, Kyoto 606-8502, Japan}
\email{narutaka@kurims.kyoto-u.ac.jp}

\thanks{
The work of the fist named author is partially supported by the ERC grant GroIsRan.
The work of the second named author is partially supported by
 JSPS KAKENHI Grant Number 26400114. He also expresses his gratitude to  the organizers of the programs "Classification
of operator algebras: complexity, rigidity, and dynamics'', "Von Neumann Algebras'',
the Mittag-Leffler Institute, the Hausdorff institute and Ecole Normale, Paris for
the hospitality during his visits. The authors would like to thank the referee for helpful remarks.}

\begin{document}
\maketitle
\begin{abstract}
 Given a $1$-cocycle $b$ with coefficients in an orthogonal representation, 
we show that every finite dimensional summand of $b$ is cohomologically trivial 
if and only if $\| b(X_n) \|^2/n$  tends to a constant in probability, 
where $X_n$ is the trajectory of the random walk $(G,\mu)$.
As a corollary, we obtain  sufficient conditions for $G$ to satisfy
Shalom's property $H_{\mathrm{FD}}$. Another application  is a convergence to a constant in probability of $\mu^{*n}(e) -\mu^{*n}(g)$, $n\gg m$,  normalized by its average with respect to $\mu^{*m}$, for any finitely generated infinite amenable group 
without infinite virtually abelian quotients. 
Finally, we show that the harmonic equivariant mapping of $G$ to a Hilbert space  obtained as an $U$-ultralimit of normalized $\mu^{*n}- g \mu^{*n}$ can depend 
on the ultrafilter $U$ for some groups. 
\end{abstract}

\section{introduction}

\subsection*{Convention}
Throughout the paper, $G$ is a compactly generated locally compact group 
with a distinguished relatively compact symmetric subset $Q$ which contains an open generating 
neighborhood $e$ of $G$, 
and $\mu$ is a symmetric probability measure on $G$ that satisfies 
the following conditions:
\begin{itemize}
\item
$\mu$ is absolutely continuous with respect to the Haar measure $m$,
\item
$\inf\{  \frac{d\mu}{dm}(x) : x\in Q\}>0$, 
\item
$\int |x|_G^d\,d\mu(x)<\infty$ for all $d$. 
\end{itemize}
Here $|x|_G := \min\{ n : x \in Q^n\}$ (except that $|e|_G:=0$). 
Note that $|\,\cdot\,|_G$ is a length function, that is, it satisfies 
$|x|_G=|x^{-1}|_G$ and  $|xy|_G\le |x|_G+|y|_G$.
Put $B_G(r) := \{ x\in G : |x|_G\le r\}$.

\subsection*{Formulation of the results}
Throughout the paper, we will work with real Hilbert spaces and orthogonal representations. 
This is purely for our convenience and all results (but not the proofs) hold true for 
complex Hilbert spaces and unitary representations (except that 
the statement of Theorem~\ref{theorem:martingaleclt} has to be slightly 
modified), because any complex Hilbert space $\cH_{\IC}$ 
is also a real Hilbert space with the real inner product $(v,w)\mapsto\Re\ip{v,w}_{\cH_{\IC}}$, 
and any $1$-cocycle (defined below) with coefficients in a unitary representation can 
be regarded as the one with coefficients in the corresponding orthogonal representation.

Let $\pi\colon G\acts\cH$ be an orthogonal representation on a real Hilbert 
space $\cH$.
Recall that a \emph{$1$-cocycle} (or simply a cocycle) 
is a continuous map $b\colon G\to\cH$ which satisfies 
the $1$-cocycle identity: $b(gx)=b(g)+\pi_gb(x)$ for all $g,x\in G$. 
It is a \emph{$1$-coboundary} if there is $v\in\cH$ 
such that $b(x)=v-\pi_xv$ for all $x\in G$. 
We note that $b$ is a $1$-coboundary if and only if it is bounded on $G$ 
(Proposition 2.2.9 in \cite{bhv}).
Every cocycle $b$ satisfies that $b(e)=0$ and 
$\| b(x)-b(y) \|=\| b(x^{-1}y)\|\le \|b\|_Q|x^{-1}y|_G$, 
where $\| b\|_Q :=\sup_{g\in Q}\| b(g) \|<\infty$. 

A cocycle $b$ is said to be \emph{$\mu$-harmonic} (or simply harmonic) 
if $\int b(gx)\,d\mu(x)=b(g)$ for all $g$, or equivalently $\int b(x)\,d\mu(x)=0$. 
Any cocycle $b$ gives rise to an affine isometric action 
$A\colon G \times\cH \to \cH$ by $A(g,v)=\pi_gv+b(g)$ (see Chapter~2 in \cite{bhv}).
Conversely, for any (affine) isometric action on a Hilbert space and 
a point $v\in\cH$, the map $b(g)=A(g,v)-v$ defines a $1$-cocycle, 
and harmonicity of this cocycle is same as harmonicity of 
the orbit map $g\mapsto A(g,v)$. 
Under an appropriate assumption on the decay of a non-degenerate measure  $\mu$, it is known that a compactly generated locally compact group $G$ admits a non-zero $\mu$-harmonic cocycle with respect to some orthogonal representation  if and only if $G$ does not  satisfy Kazhdan' property (T).
Existence of a non-zero harmonic cocycle  on 
groups which do not satisfy property (T) is proved by Mok (\cite[Cor. 0.1]{mok}), Korevaar and Schoen  \cite[Thm 4.1.2]{ks} for finitely presented groups (and not discrete definition of harmonicity)  and in general case (and discrete definition of harmonicity) by Shalom in  \cite[Thm 6.1]{shalom-invent}. We will give somewhat more constructive proof of 
this fact in Section~\ref{sec:bmu}.
 See also Gromov \cite[Section 3.6]{gromovrandomwalk},  \cite[Section 7A]{gromovasymptotic}  Fisher and Margulis \cite{fishermargulis}, 
Lee and Peres \cite[Thm 3.8]{leeperes},  Ozawa \cite{ozawa}  as well as the book by Bekka, de la Harpe, and Valette \cite{bhv} for a non-exhaustive list of references about this result.

We say that a $1$-cocycle $b$ is \emph{finite-dimensional} if 
the $\pi(G)$-invariant subspace $\cspan b(G)$ is finite-dimensional. 
If $\cH=\bigoplus_i\cH_i$ is some orthogonal decomposition of $\cH$ into $\pi(G)$-invariant 
subspaces, then $b=\bigoplus_i P_{\cH_i}b$ is a decomposition of $b$ into $1$-cocycles $P_{\cH_i}b$ 
(with respect to $\pi|_{\cH_i}$). We call each $P_{\cH_i}b$ a summand of $b$. 
We say that such summand is {\it cohomologically trivial} if it is a $1$-coboundary. 

Given a probability measure $\mu$ on $G$, 
let $X_n$ denote the trajectory of the random walk  $(G,\mu)$, that is, 
$X_n=s_1s_2\cdots s_n$ where increments $s_i \in G$ are independent and chosen with respect to $\mu$.
The corresponding probability measure and its expectation are 
denoted by $\IP$ and $\IE$.

The value of a Hilbert valued $\mu$-harmonic $1$-cocycle 
along a trajectory of the random walk $(G,\mu)$ is a martingale, 
and therefore 
\[
\ie{ \|b(X_n)\|^2 }= \sum_{k=1}^n \ie{ \|b(X_k)\|^2 - \|b(X_{k-1})\|^2 } = n \ie{ \|b(X_1)\|^2}.
\]
That is, the expected value $\frac{1}{n}\ie{\|b(X_n)\|^2}$ is equal to a constant, not depending on $n$. 
For any (not necessarily harmonic) $1$-cocycle $b$, the expected value $\frac{1}{n}\ie{\|b(X_n)\|^2}$ 
has a limit (see Lemma \ref{lem:harmonicpart}).
Theorem~\ref{thm:beta} below characterizes the case when the random variable 
$\frac{1}{n}\|b(X_n) \|^2$ tends to a constant.

\begin{maintheorem} \label{thm:beta}
Let $G$ be a compactly generated locally compact group with a probability measure $\mu$ 
on $G$ as in Convention. Let $b\colon G\to\cH$ be a  $1$-cocycle.  
Then the following conditions are equivalent:
\begin{enumerate}
\item\label{item:beta1} Any finite-dimensional summand of $b$ is cohomologically trivial. 
\item\label{item:beta2} $\frac{1}{n}||b(X_n)||^2$ tends to a constant  in probability. 
\end{enumerate}
Now assume moreover that $b$  is harmonic and put $c = \int_G \| b(x) \|^2\,d\mu(x)$.
Then the limit
\[
\beta := \lim_{n\to\infty} \frac{1}{2c^2}\ie{\bigl| \frac{\|b(X_n)\|^2}{n} - c \bigr|^2}
\]
always exists, and $\beta=0$ if and only if $(\ref{item:beta1})$ and $(\ref{item:beta2})$ hold. 
If $\beta \ne 0$, then $b$ has a cohomologically non-trivial finite-dimensional summand 
of dimension $\le 1/\beta$.
\end{maintheorem}

A more precise version of Theorem \ref{thm:beta} will be given in Theorem~\ref{theorem:martingaleclt}, where we describe  the limit distribution of $\| b(X_n ) \|/\sqrt{n}$.
This theorem has the following corollary:

\begin{corollaryint}
Let $b$ be a harmonic cocycle. Then, 
$b$ is a direct sum of (possibly infinitely many) finite-dimensional cocycles
if and only if 
$\limsup_n \IP( \| b(X_n) \| < c\sqrt{n} ) > 0$
for every $c>0$.
\end{corollaryint}

Recall that a group $G$ is said to have \emph{Shalom's property 
$H_\mathrm{FD}$} if every orthogonal representation $\pi$ with 
non-zero reduced cohomology group $\overline{H^1}(G,\pi)$ 
contains a non-zero finite-dimensional subrepresentation. 
In Corollary \ref{cor:beta} we show that $G$ satisfies Shalom's property 
$H_\mathrm{FD}$ if at least one of the two following conditions hold:
either $\liminf_n\| \mu^{*n} - \mu^{*(1+\delta)n} \|_1<2$ for some $\delta>0$ or  $\limsup_n\mu^{*n}(B_G(c\sqrt{n}))>0$ for all $c>0$.

Theorem \ref{thm:beta} and its corollaries develop the argument from \cite{ozawa}.
While the main result of \cite{ozawa} is a new proof of Gromov's polynomial growth theorem, the paper also provides 
a more general  criterion for the property $H_{FD}$ for a finitely generated group 
in terms of convolutions 
of random walks is given in Section~$4$ of \cite{ozawa}.  It is shown in \cite{erschler} that wreath products of $\Z$ with finite groups satisfy the assumption of that criterion, providing examples of groups of super-polynomial growth where the criterion applies.
The assumption of the criterion  from Section~$4$  in \cite{ozawa} uses shifted convolution, 
and it is not clear whether this assumption is defined 
by an unmarked Cayley graph of $G$. 
Assume that $(G,\mu)$ is a simple random walk on $G$, that is, 
$\mu$ is equidistributed on a finite generating set of $G$.
The conditions of $(1)$ as well as of $(2)$ of Corollary \ref{cor:beta} 
are clearly defined by the unmarked Cayley graph of $G$. 
We do not know  any group which satisfies the assumption of (1) or of (2) of Corollary \ref{cor:beta}
and for which we know that it violates the assumption of Section 4 of \cite{ozawa}.  
But the conditions of  Corollary \ref{cor:beta}
are easier to check than the assumption from \cite{ozawa}. 
For example, it is easily applicable to solvable Baumslag--Solitar groups,  lamplighter groups $\IZ\ltimes\bigoplus_{\IZ} F$ with $F$ finite,   or to polycyclic groups obtained as extension of ${\IZ}^2$ by $M \in SL(2,d)$ with eigenvalues of absolute value $\ne 1$. See Section \ref{sec:appli} for more examples.
We do not know any group which satisfies Shalom's property and does not satisfy the assumption of  Corollary \ref{cor:beta}.

Given a not necessarily harmonic cocycle $b$ on a group without property (T), a harmonic cocycle can be obtained taking averages of $b$ (see Mok, Korevaar Schoen, Shalom \cite{mok,ks, shalom}, and in particular this can be achieved averaging with respect to a  probability measure $\mu$ (see e.g. Gromov, Lee--Peres \cite{gromovrandomwalk, leeperes}). In Section \ref{sec:bmu} we study the cocycles $b_{\mu,U}$, constructed  as a ultralimit in $\ell_2(G)$ of normalized $\mu^{*n}- g \mu^{*n}$ 
on a finitely generated amenable group $G$. 
Kesten's criterion \cite{kesten} (see also  \cite{avez}) 
implies that $\mu^{*n}$ is a sequence of almost invariant vectors in $\ell_2(G)$, and  
one can moreover show (see Theorem \ref{thm:bmu}) that the limit is a harmonic $1$-cocycle.
Applying Theorem \ref{thm:beta} to this $1$-cocycle, one obtains

\begin{maintheorem}\label{theoremcorollary}
Let $G$ be a finitely generated infinite amenable group without virtually 
abelian infinite quotients. 
Let $\mu$ be a finitely-supported  symmetric non-degenerate probability measure.
Then $(\mu^{*2n}(e)  -\mu^{*2n}(X_{2m}))/\alpha(m,n)$ tends to a constant 
in probability $\mu^{*2m}$ as $m\to\infty$ and $n\gg m$. 
Here $\alpha(m,n)=  \mu^{*2n}(e) - \mu^{*2n+2m}(e)$ 
is the average of $\mu^{*2n}(e)  -\mu^{*2n}(g)$
with respect to $\mu^{*2m}$. Namely 
\[
\lim_{m\to\infty}\limsup_{n\to\infty}\IE\left|\frac{\mu^{*2n}(e) - \mu^{*2n}(X_{2m})}{\mu^{*2n}(e) - \mu^{*2n+2m}(e)} - 1 \right| = 0.
\]
\end{maintheorem}

Take $n$ much larger than $m$. Observe that a  group  is amenable if and only if $\mu^{*2n}(g)/\mu^{*2n}(e)$ is close to $1$ in probability with respect to $\mu^{*2m}$.  Theorem~\ref{theoremcorollary} gives a sufficient condition for the concentration of the second order term of $\mu^{*2n}$.

Theorem~\ref{theoremcorollary} applies in particular to 
any finitely generated amenable torsion group (such as
Grigorchuk groups $G_w$ \cite{grigorchukdegrees}) or 
to any finitely generated amenable simple group (such as commutator full topological groups of minimal shifts on $\Z$ (which are simple by a result of Matui \cite{matui} and amenable  by a result of Juschenko--Monod \cite{juschenkomonod}), or to simple groups of intermediate growth constructed recently by Nekrashevych \cite{nekrashevych}.
If $\mu$ is equidistributed on a finite generating set of $G$, then the assumption of Theorem \ref{theoremcorollary} depends only on the unmarked Cayley graph of $(G,\mu)$.
In particular, the theorem gives a necessary condition for an amenable  group to be simple in terms of  unmarked Cayley graphs. In general, it is known that the property of being simple can not be defined by the unmarked Cayley graphs, as it is shown by Burger and Mozes \cite{burgermozes} (their examples are isometric to product of two trees and they are non-amenable). It is to our knowledge an open problem whether a property of being a torsion group can be verified geometrically.

Geometric group theory tries to recover properties of a group from the word metrics of this group. Given a group $G$, generated by a finite set $S$, its action on a metric space $X$ and a  point $x_0\in X$, the group $G$ is equipped with two metrics: the word metric $d_{G,S}(g,h)$ as well as $d_{X,x_0}(g,h)= d_X(gx_0, hx_0)$. It seems interesting to study which properties of the action, or of the group $G$, can be recovered from these two metrics. Theorem A as well as Corollary 2.5 provide examples of such situation, for $X$ being a Hilbert space and a group $G$ acting by affine transformations of $X$.

Fix a non-principal ultrafilter $U$ on the natural numbers $\IN$.
Let $b^{p,q}_{\mu,U}$ be the mapping to a vector space equipped with a metric, 
constructed as $U$ ultralimit of normalized  $ (\mu^{*n})^q - g(\mu^{*n})^q$, 
considered as elements of $\ell_p(G)$ (see Section~\ref{sec:pthin}). 
This means that  we divide  $g (\mu^{*n})^q - (\mu^{*n})^q$  by  the $l_p$ norm of this expression, considered as a function on $g$, and then we take the ultralimit with respect to $U$.
By the construction,  the $l_p$ norm of  $b^{p,q}_{U,\mu,G}$ is one.
We recall that any ultralimit of Hilbert spaces is a Hilbert space, so that for $p=2$ and any $q\ge 0$ we obtain a cocycle with respect to some orthogonal 
representation of $\cH$. 
 In particular, for $q=1$ and $p=2$ , $b^{p,q}_{\mu,U}$ coincides up to a multiplicative constant with  the harmonic cocycle  $b_{\mu,U}$, studied in the proof of Theorem B in Section \ref{sec:bmu}.
In general, for $p\ne 2$, we obtain a cocycle with respect to some isometric 
representation on an abstract $L_p$-space.   

 In Theorem \ref{dependencesubsequence} below we show that the  cocycles   $b^{p,q}_{\mu,U}$,  $p\ge 1$, $q\ge 0$ (in particular, the harmonic cocycle $b_{\mu,  U}$)  can depend on the choice of a non-principal ultrafilter $U$.

\begin{maintheorem}  \label{dependencesubsequence}
Take $p=1$ or $2$ and $q=0$, $1$, or $2$.
 For any $D\ge 2$ there exist torsion groups $G_1, G_2, \ldots, G_D$ such that the following holds. Consider finitely supported symmetric non-degenerate measures $\mu_i$ on $G_i$ and put $G= \prod_{j=1}^D G_i$ and $\mu= \prod_{j=1}^D \mu_i$.
For each $j=1,\ldots,D$
there exists a non-principal ultrafilter $U$ such that 
the limiting cocycle $b^{p,q}_{\mu,U}$ factors through $G\twoheadrightarrow G_i$.  
\end{maintheorem}

Theorem    \ref{dependencesubsequence} shows 
in particular that  there exist at least $D$ mutually distinct limiting cocycles among $\{b^{p,q}_{\mu,U} : U \}$, and at least $D$ mutually distinct subgroups among possible  kernels of such cocycles. Such groups $G$ admit $g_1, g_2 \in G$ such that the ratio
$(\mu^{*2n}(e)-\mu^{*2n}(g_1))/ (\mu^{*2n}(e) -\mu^{*2n} (g_2))$ does not have a limit as $n \to \infty$.

The groups $G_i$ are constructed as piecewise automatic groups \cite{erschlerpiecewise}, they can be chosen to be of sub-exponential word growth, but in such a way that for each $j$ the group $G_j$ is in some sense very close to a non-amenable group on some scale while on this particular scale it does not happen to other $G_k$, $j\ne k$. The contribution to $b^{p,q}_{\mu,U}$ is mainly from $G_j$ on this scale, and  the kernel of $b^{p,q}_{\mu,U}$ contains $\prod_{k\ne j} G_k$.

The kernels of cocycles $b^{p,q}_{\mu,U}$ are particular cases of what we call $\ell_p$-thin subgroups: this is a natural family of subgroups, related to the shifts $(\mu^{*n})^q$
(see Definition \ref{definitionsubgroup}), which for $p=2$, $q=1$ is related to amenability, for $p=q=1$ to Poisson-Furstenberg boundary and for $q=0$, $p\ge 1$ to growth of groups
(see  Lemma \ref{directproducts}), these groups in some situation may depend  on $p$  (see Example \ref{dependancep}) and on the measure $\mu$ (see Remark \ref{rem:notnormal}).

Since the group $G$ in the statement of the theorem is a torsion group, it does not admit a virtual quotient to an infinite cyclic group. In particular, taking $p=2$ we can apply the conclusion of Theorem \ref{theoremcorollary}
 to 
$(G,\mu)$ to claim that $\mu^{*n}(e) -\mu^{*n}(g)$, normalised by its average $\alpha(m,n)$ is close to a constant in probability $\mu^{*m}$, for $n \gg m$. In other words,  for each $n\gg m$ $\mu^{*m}$ is concentrated on a set where
normalized $\mu^{*n}(e) -\mu^{*n}(g)$ is close to its mean value, but in view of \ref{dependencesubsequence}  these sets may depend essentially on $n$.

We are grateful to Pierre de la Harpe for comments on the preliminary version of this paper.

\section{Harmonic cocycles and finite-dimensional summands}\label{sec:harmonic}

We now recall from Sections 4 and 5 in \cite{guichardet} 
that the space $Z^1(G,\pi)$ of $1$-cocycles is a Hilbert space
under the norm 
\[
\| b \|_{L^2(\mu)} := \Bigl(\int_G \| b(x) \|^2\,d\mu(x)\Bigr)^{1/2},
\]
and it decomposes into an orthogonal direct sum of
approximate $1$-coboundaries and $\mu$-harmonic $1$-cocycles.
We will say $b$ is \emph{normalized} when $\| b \|_{L^2(\mu)}=1$. 

\begin{lemma} \label{lemmaequivnorms}
The space $Z^1(G,\pi)$ of $1$-cocycles is a Hilbert space with respect 
to the norm $\|\,\cdot\,\|_{L^2(\mu)}$. 
Moreover the norms  $\|\,\cdot\,\|_{L^2(\mu)}$ and  $\|\,\cdot\,\|_Q$ are equivalent.
\end{lemma}
\begin{proof}
We observe that $Z^1(G,\pi)$ is a Banach space w.r.t.\ 
the norm $\|\,\cdot\,\|_Q$ (see \cite[Chapter~3]{bhv}), 
and that $\| b \|_{L^2(\mu)}\le(\int |x|_G^2\,d\mu(x))^{1/2} \| b \|_Q$. 
The other side inequality follows, via Open Mapping Theorem, from the fact that any 
measurable locally integrable $1$-cocycle into a separable Hilbert space 
is automatically continuous modulo a null set. 
However, following \cite{guichardet}, we give a more direct proof here. 
Take an open generating neighborhood $U$ of $e$ such that 
$U\subset Q$ and an open neighborhood $V$ of $e$ such that $V^2\subset U$. 
We observe that $(\int\|b(x)\|^2\,d\mu^{*2}(x))^{1/2}\le 2\|b\|_{L^2(\mu)}$
and $\ve:=\inf_{x\in UV}\frac{d\mu^{*2}}{dm}(x)>0$.
Thus, for every $g\in U$ one has 
\begin{align*}
\| b(g) \|^2 &= m(V)^{-1}\int_V \| b(gx) - \pi_gb(x) \|^2\,dm(x)\\
&\le 2m(V)^{-1}[\int_{gV}\|b(x)\|^2\,dm(x) + \int_{V}\|b(x)\|^2\,dm(x) ]\\
&\le 4\ve^{-1}m(V)^{-1}\int_{UV}\|b(x)\|^2\,d\mu^{*2}(x)\\
&\le 16\ve^{-1}m(V)^{-1}\|b\|_{L^2(\mu)}^2.
\end{align*}
Since there is $N\in\IN$ such that $Q\subset U^N$, 
this proves that the norms $\|\,\cdot\,\|_{L^2(\mu)}$ and $\|\,\cdot\,\|_Q$ are equivalent, 
and that $Z^1(G,\pi)$ is a Hilbert space w.r.t.\ the norm $\|\,\cdot\,\|_{L^2(\mu)}$.
\end{proof}

The \emph{reduced $1$-cohomology} group 
$\overline{H^1}(G,\pi):=Z^1(G,\pi)/\overline{B^1(G,\pi)}$ is 
defined to be the space $Z^1(G,\pi)$ of $1$-cocycles modulo the closure 
of the subspace $B^1(G,\pi)$ of $1$-coboundaries. 
We note that $\overline{B^1(G,\pi)}=B^1(G,\pi)$ if $\pi$ is finite-dimensional, 
by Theorem~1 in \cite{guichardet}.
See Chapter~3 in \cite{bhv} for an introduction to
first reduced cohomology groups. Thus, 
\[
Z^1(G,\pi)=\overline{B^1(G,\pi)}\oplus B^1(G,\pi)^\perp
\ \mbox{ and }\ 
\overline{H^1}(G,\pi) \cong B^1(G,\pi)^\perp.
\] 
We observe that 
$b\in Z^1(G,\pi)$ belongs to $B^1(G,\pi)^\perp$ if and only if 
it is \emph{$\mu$-harmonic} in the sense $\int b(x)\,d\mu(x)=0$ 
or equivalently $\int b(gx)\,d\mu(x) = b(g)$ for all $g\in G$.
Indeed, this follows from the identities $b(x^{-1})+\pi_x^{-1}b(x)=b(e)=0$ and 
\[
\int \ip{b(x),v-\pi_x v}\,d\mu(x) = 2\ip{\int b(x)\,d\mu(x), v}.
\]
We note that every summand of a $\mu$-harmonic $1$-cocycle 
is $\mu$-harmonic and that every non-zero $\mu$-harmonic $1$-cocycle 
is not a $1$-coboundary. 

We recall the general fact about orthogonal representations. 
Let $(\pi,\cH)$ be an orthogonal representation of $G$ and put 
\[
T_0:=\ie{\pi(X_1)}=\int\pi(g)\,d\mu(g).
\] 
Then, $T_0$ is a self-adjoint contraction on the Hilbert space $\cH$ 
such that $T_0^k=\ie{\pi(X_k)}$ for every $k$.
By strict convexity of a Hilbert space, a vector $v\in\cH$ 
satisfies $T_0v=v$ if and only if $\pi_g v=v$ for $\mu$-a.e.\ $g$, 
which is equivalent to that $v$ is $\pi(G)$-invariant. 
Thus by spectral theory, the operators $\frac{1}{n}\sum_{k=0}^{n-1}T_0^k = \ie{\frac{1}{n}\sum_{k=0}^{n-1}\pi(X_k)}$ 
converge in strong operator topology to the orthogonal projection $P_0$ 
onto the subspace of $\pi(G)$-invariant vectors. 
One moreover has convergence in probability 
\[
\| \frac{1}{n}\sum_{n=0}^{n-1}\pi(X_k)v - P_0 v \| \stackrel{\IP}{\to} 0.
\]
Indeed, to prove it, one may assume $P_0=0$ and in this case 
\[
\ie{\|\frac{1}{n}\sum_{n=0}^{n-1}\pi(X_k)v\|^2} = \frac{1}{n^2}\sum_{k,l=0}^{n-1} \ip{T_0^{|k-l|}v,v} \to 0.
\]

\begin{lemma}\label{lem:harmonicpart}
For every $b\in Z^1(G,\pi)=\overline{B^1(G,\pi)}\oplus B^1(G,\pi)^\perp$, 
one has
\[
\lim_n\frac{1}{n}\ie{\|b(X_n)\|^2}=\| b_{\mathrm{harm}} \|_{L^2(\mu)}^2,
\]
where $b_{\mathrm{harm}}$ is the $B^1(G,\pi)^\perp$ summand in the
above decomposition. 
In particular, $b$ is nonzero in $\overline{H^1}(G,\pi)$ if and only if
$\lim \frac{1}{n}\ie{\|b(X_n)\|^2}>0$.
\end{lemma}
\begin{proof}
Let $T_0:=\int \pi(g)\,d\mu(g)$.
If $c\in B^1(\pi,\cH)$ is a $1$-coboundary, $c(x) = v - \pi_x v$,
then for every $n$ one has
\[
\frac{1}{n} \ie{ \| c(X_n) \|^2} = \frac{2}{n}\ip{(1-T_0^n)v,v}
 \le 2\ip{(1-T_0)v,v} = \| c \|_{L^2(\mu)}^2.
\]
Since $c\mapsto \ie{ \| c(X_n) \|^2}$ is norm-continuous by Lemma \ref{lemmaequivnorms},
the above inequality holds for all $c\in \overline{B^1(G,\pi)}$.
Hence, for any $c\in \overline{B^1(G,\pi)}$, by approximating it
by $c_m\in B^1(G,\pi)$, one has
\[
\limsup_n \frac{1}{n} \ie{ \| c(X_n) \|^2}
=\limsup_n \frac{1}{n} \ie{ \| (c-c_m)(X_n) \|^2}
\le \| c-c_m \|_{L^2(\mu)}^2 \to 0.
\]
Now let $b=c+b_{\mathrm{harm}}\in\overline{B^1(G,\pi)} + B^1(G,\pi)^\perp$
be given.
Note that since $b_{\mathrm{harm}}$ is $\mu^{*n}$-harmonic, it is orthogonal to
$c$ in $L^2(\mu^{*n})$. Consequently, one has
\smallskip

\hspace*{\fill}$\displaystyle
\lim_n\frac{1}{n}\ie{\|b(X_n)\|^2}
 = \lim_n \frac{1}{n}\ie{\|c(X_n)\|^2+\|b_{\mathrm{harm}}(X_n)\|^2}
 =  \| b_{\mathrm{harm}} \|_{L^2(\mu)}^2.$\hspace*{\fill}
\end{proof}

It is not clear whether $\frac{1}{n^2}\ie{\|b(X_n)\|^4}$ is bounded 
for every $1$-cocycle $b$. However, it is the case for any $\mu$-harmonic $1$-cocycle $b$ 
(cf.\ Footnote 2 in \cite{leeperes}).
\begin{lemma}\label{lemma:aprioribound}
For every $d$, one has 
\[
\sup_n\sup_{b} \frac{1}{n^d}\ie{\|b(X_n)\|^{2d}} < \infty,
\]
where the supremum runs over all normalized $\mu$-harmonic $1$-cocycles $b$.
\end{lemma}
\begin{proof}
We fix a universal orthogonal representation $(\pi,\cH)$ and 
consider the operators $U_n$ from the space of $\mu$-harmonic 
cocycles into $L^{2d}(\mu^{*n};\cH)$, given by 
$U_nb = n^{-1/2}b$. 
Since
\begin{align*}
\| U_n b \|  = (\frac{1}{n^d}\ie{ \|b(X_n)\|^{2d} } )^{1/2d} 
\le n^{1/2}\ie{ |X_1|_G^{2d} }^{1/2d}\|b\|_Q
\end{align*}
(by the H\"older inequality $(\sum_{i=1}^n a_i)^{2d} \le n^{2d-1}\sum_{i=1}^n a_i^{2d}$ 
for $a_i\geq0$),
the operators $U_n$ are bounded by Lemma~\ref{lemmaequivnorms}.
The lemma claims that $U_n$'s are uniformly bounded. 
For this, by Principle of Uniform Boundedness, 
it suffices to show $\sup_n \|U_nb\|<\infty$ for each $b$.
(The use of PUB can be avoided if one does the following proof more meticulously.) 
We in fact prove that $\limsup_n \frac{1}{n^d}\ie{\|b(X_n)\|^{2d}} \le (2d-1)!!$ 
for each normalized harmonic cocycle $b$, by induction on $d$. 
Here $(2d-1)!!=\prod_{k=1}^d (2k-1)$. 
The case $d=1$ is clear. 
By induction hypothesis and the Cauchy--Schwarz inequality when $k$ is odd, 
we may assume that there is $C>0$ such that 
$\ie{\|b(X_n)\|^{k}} \le Cn^{k/2}$ for all $k\le 2(d-1)$. 
It follows that 
\begin{align*}
\ie{\|b(X_n)\|^{2d}} &=\iint \| b(x) - b(y) \|^{2d}\,d\mu^{*n-1}(x)d\mu(y) \\
 &\hspace*{-50pt}= \iint (\|b(x)\|^2-2\ip{b(x),b(y)}+\|b(y)\|^2)^d\,d\mu^{*n-1}(x)d\mu(y) \\
 &\hspace*{-50pt}= \iint \| b(x) \|^{2d} + \binom{d}{1}\| b(x) \|^{2(d-1)}\|b(y)\|^2 \\ 
   &\hspace*{-20pt} + 4\binom{d}{2}\| b(x)
\|^{2(d-2)}|\ip{b(x),b(y)}|^2\,d\mu^{*n-1}(x)d\mu(y) + C'n^{(2d-3)/2} \\
 &\hspace*{-50pt}\le \ie{\|b(X_{n-1})\|^{2d}} + (d+2d(d-1))\cdot (2d-3)!!\cdot n^{d-1}+ C'n^{d-3/2} \\
 &\hspace*{-50pt}\le\cdots\le \sum_{k=1}^n ((2d-1)!!\cdot dn^{d-1}+ C'k^{d-3/2}) \\
 &\hspace*{-50pt}=(2d-1)!!\cdot n^d + o(n^d),
\end{align*}
where $C'$ is some constant depending on $d$ but not on $n$. 
This finishes the proof.
\end{proof}

We start the proof of Theorem~\ref{thm:beta}. 
Recall that the tensor product Hilbert space $\cH\otimes\cH$ is canonically 
identified with the space of Hilbert--Schmidt operators $\cS_2(\cH)$ on $\cH$ 
via $v'\otimes v \leftrightarrow S_{v'\otimes v}$, 
where $S_{v'\otimes v}(u)=\ip{u,v}v'$. 
Under this identification, the operators 
$\pi_g\otimes\pi_g$ on $\cH\otimes\cH$ act on $\cS_2(\cH)$ by 
conjugation $\Ad\pi_g\colon S\mapsto \pi_gS\pi_g^*$. 
Every Hilbert--Schmidt operator is compact and 
every compact self-adjoint operator $S$ has a unique spectral 
decomposition $S=\sum_i \lambda_i E_i$ where $\lambda_i\in\IR$ 
are the non-zero eigenvalues of $S$ and $E_i$ are the finite-rank orthogonal projections 
onto the corresponding eigenspaces. 
If $v\in\cH\otimes\cH$ is $(\pi\otimes\pi)(G)$-invariant, then 
$S_v$ is $\Ad\pi(G)$-invariant and so are the spectral projections 
$E_i$'s, which means that $E_i\cH$ are finite-dimensional $\pi(G)$-invariant subspaces. 

Now let us consider a $1$-cocycle $b\colon G\to\cH$ and put 
\[
w:=\int (b\otimes b)(x)\,d\mu(x) \in \cH\otimes\cH 
\ \mbox{ and }\ 
T:=\int \pi_g\otimes \pi_g \,d\mu(g).
\]
Then, $T$ is a self-adjoint contraction on $\cH\otimes\cH$, 
which is positivity preserving as an operator on $\cS_2(\cH)$.
By the previous discussion, $\frac{1}{n}\sum_{k=0}^{n-1}T^k$ converges 
in strong operator topology to the orthogonal projection $P$ from 
$\cH\otimes\cH$ onto the subspace of $(\pi\otimes\pi)(G)$-invariant vectors.
In particular, $\frac{1}{n}\sum_{k=0}^{n-1}T^kw$ converges to $Pw$ in norm and 
$S_{Pw}$ is a positive Hilbert--Schmidt operator which is $\Ad\pi(G)$-invariant. 
For any $\pi(G)$-invariant closed subspace $\cK\subset\cH$, one has 
\[
P_{\cK} S_{Pw} P_{\cK} = S_{(P_{\cK}\otimes P_{\cK})Pw} = S_{P(P_{\cK}\otimes P_{\cK})w}=S_{Pw_{\cK}}, 
\] 
where $w_{\cK}=\int (b_{\cK}\otimes b_{\cK})(x)\,d\mu(x)$ for the cocycle $b_{\cK}=P_{\cK}b$. 
If $b$ is finite-dimensional, then the trace $\Tr$ is norm-continuous and 
\[
\Tr(S_{Pw})=\Tr(\lim_n \frac{1}{n}\sum_{k=0}^{n-1}  S_{T^kw}) = \Tr(S_{w}) = \|b\|_{L^2(\mu)}^2.
\]
In general, one has 
the spectral decomposition 
\[
\tag{$\ast$} S_{Pw}=\sum_i\lambda_i E_i
\] where 
$\lambda_1,\lambda_2,\ldots$ is a finite or infinite sequence of 
strictly positive numbers 
and $E_i\cH$'s are finite-dimensional $\pi(G)$-invariant subspaces. 
Thus for $b_i:=E_ib$ and $b_\infty:=b - (\sum_i b_i)$, 
one has the direct sum decomposition $b=b_\infty+\sum_i b_i$. 
We claim that each $1$-cocycle $b_i$, $i\neq\infty$, is nonzero and that 
$b_\infty$ is \emph{weakly mixing} in the sense that it does not admit 
a nonzero finite-dimensional summand anymore. 
First, put $E_\infty:=1-\sum_iE_i$ and observe that 
for $w_i:=(E_i\otimes E_i)w=\int (b_i\otimes b_i)(x)\,d\mu(x)$, 
one has $S_{Pw_i}=E_iS_{Pw}E_i=\lambda_iE_i$, 
including the case $i=\infty$ and $\lambda_\infty:=0$. 
It follows that 
$b=b_\infty\oplus\sum_i^\oplus b_i$ and 
$S_{Pw}=S_{Pw_\infty}\oplus\sum_i^\oplus S_{Pw_i}$ 
in accordance with $\cH=E_\infty\cH\oplus\bigoplus_i E_i\cH$.
That $S_{Pw_\infty}=0$ means that $b_\infty$ is weakly mixing. 
Thus $\|Pw\|\neq0$ if and only if $b$ has a nonzero finite-dimensional summand.
Moreover, one has 
\[
\Tr(S_{Pw}) = \sum_i \Tr(S_{Pw_i}) = \sum_i \lambda_i\Tr(E_i) = \sum_i \|b_i\|_{L^2(\mu)}^2
\]
and 
\[
\|Pw\|^2=\Tr(S_{Pw}^2)=\sum_i\lambda_i^2 \Tr(E_i).
\]

For the proof of Theorem~\ref{thm:beta}, 
in view of Lemma \ref{lem:harmonicpart} and the fact that any 
nonzero $\mu$-harmonic $1$-cocycle is cohomologically non-trivial, 
we may assume that the $1$-cocycle $b$ is $\mu$-harmonic.
For such $b$, we have the following more precise form of Theorem~\ref{thm:beta}.

For any $\theta\ge 0$ and any finite or infinite (possibly null) 
sequence $\sigma_k$ of positive numbers, 
we denote by $\chi(\theta, \sigma_k)$ the distribution 
of $\sqrt{ \theta^2 + \sum_k \sigma_k^2 g_k^2 }$,
where $g_k$ are independent standard centered Gaussian random variables.


\begin{theorem}\label{theorem:martingaleclt}
Let $G$ be as in Convention. Let $b$ be a normalized $\mu$-harmonic $1$-cocycle.
Let $w$, $Pw$, and $S_{Pw}=\sum_i\lambda_i E_i$  be as defined  in $(*)$  before the formulation of the theorem.  
Then, 
\[
\lim_{n\to\infty} \frac{1}{2}\ie{\bigl| \frac{\|b(X_n)\|^2}{n} - 1 \bigr|^2} = \|Pw\|^2 \le (\min_i \dim E_i\cH)^{-1}.
\]

Moreover, 
the random variables $\frac{1}{\sqrt{n}}\| b(X_n) \|$ converge in distribution and in moments to
$\chi(\theta, \sigma_k)$, 
where $\theta=\| E_\infty b\|_{L^2(\mu)}$, and $\sigma_k^2$ are 
positive eigenvalues of $S_{Pw}$ counted with multiplicities 
(i.e., $\sigma_k=\lambda_i^{1/2}$ for $\sum_{l=1}^{i-1}\dim E_l\cH < k \le \sum_{l=1}^i\dim E_l\cH$), which satisfy 
$\theta^2+\sum_k\sigma_k^2=\| b \|_{L^2(\mu)}^2=1$.   
One has $\theta>0$ if and only if $b$ admits a weakly mixing summand; and 
$\sigma_k>0$ for some $k$ if and only if $b$ admits a non-zero finite-dimensional summand.
\end{theorem}
\begin{proof}[Proof of Theorem~\ref{thm:beta} and Theorem~\ref{theorem:martingaleclt}]
Let $b$ be a normalized $\mu$-harmonic $1$-cocycle. 
In the discussion above, we already saw
$\|Pw\|\neq0$ if and only if $b$ has a nonzero finite-dimensional summand. 
Moreover the above formula implies 
\[
\|Pw\|^2 = \sum_i\lambda_i^2 \Tr(E_i) \le (\max_i \lambda_i)\Tr(S_{Pw}) \le (\min_i \Tr(E_i))^{-1},
\]
since $\Tr(S_{Pw})=\sum_i\lambda_i\Tr(E_i)\le1$. Note that $\Tr(E_i)=\dim E_i\cH$.

Next, we prove that $\ie{ |\frac{\|b(X_n)\|^2}{n} - 1 |^2} \to 2\ip{Pw,w}=2\|Pw\|^2$.
Recall that 
\begin{align*}
\int (b\otimes b)(x)\,d\mu^{*n}(x)
&= \iint (b\otimes b)(xy)\,d\mu^{*n-1}(x)\,d\mu(y)\\
&= \iint (b\otimes b)(x)+(\pi_x\otimes\pi_x)(b\otimes b)(y)\,d\mu^{*n-1}(x)\,d\mu(y)\\
&= \int (b\otimes b)(x)\,d\mu^{*n-1}(x) + T^{n-1}w\\ 
&= (1+T+\cdots+T^{n-1})w,
\end{align*}
and $\int \| b(x) \|^2 \,d\mu^{*n}(x)=n$ 
(see \cite{leeperes} and \cite{ozawa}). Hence 
\begin{align*}
\ie{\|b(X_n)\|^4} &= \int \|b(x)\|^4 \,d\mu^{*n}(x) \\
 &\hspace*{-50pt}= \iint (\|b(x)-b(y)\|^2)^2 \,d\mu^{*n-1}(x)\,d\mu(y)\\ 
 &\hspace*{-50pt}= \iint \bigl[\|b(x)\|^4 + 4|\ip{b(x),b(y)}|^2 + \|b(y)\|^4 + 2\|b(x)\|^2\|b(y)\|^2\bigr] \,d\mu^{*n-1}(x)\,d\mu(y)\\ 
 &\hspace*{-50pt}= \ie{\|b(X_{n-1})\|^4} + 4\ip{\sum_{k=0}^{n-2}T^k w,w} + \ie{\|b(X_1)\|^4} + 2(n-1) \\
 &\hspace*{-50pt}= 4\ip{ \sum_{k=1}^{n-1} (n-k)T^{k-1} w,w} + n\ie{\|b(X_1)\|^4} + n(n-1) \\
 &\hspace*{-50pt}\le 3n^2 + O(n).
\end{align*}
By Bounded Convergence Theorem, this implies that
\begin{align*}
\ie{ \bigl| \frac{\|b(X_n)\|^2}{n} - 1 \bigr|^2} 
 &= \ie{ \frac{1}{n^2}\|b(X_n)\|^4 - \frac{2}{n}\|b(X_n)\|^2 +1}\\
 &= \frac{4}{n^2}\ip{  \sum_{k=1}^{n-1} (n-k)T^{k-1}w,w} + \frac{1}{n}(\ie{\|b(X_1)\|^4}-1)\\
 &\to 2\ip{Pw,w}.
\end{align*}

Now since $\sup_n \ie{\bigl| \frac{1}{n}||b(X_n)||^2 -1 \bigr|^3}<\infty$ 
by Lemma~\ref{lemma:aprioribound}, the sequence $\frac{1}{n}||b(X_n)||^2$ 
tends to a constant (which is necessarily $1$) in probability if and only if 
one has $\ie{\bigl| \frac{1}{n}||b(X_n)||^2 -1 \bigr|^2}\to0$.
This completes the proof of Theorem~\ref{thm:beta} 
and the first part of Theorem~\ref{theorem:martingaleclt}. 

For the second half of Theorem~\ref{theorem:martingaleclt}, we first note that 
convergence in distribution and convergence in moments are equivalent 
in our setting. 
Indeed, by the moments condition 
$\sup_n\frac{1}{n^d}\ie{\|b(X_n)\|^{2d}} < \infty$ 
(Lemma~\ref{lemma:aprioribound}), 
convergence in distribution implies that in moments 
(see \cite[Corollary 25.12]{billingsley}). 
And conversely, since the normal distribution and 
the distributions $\chi(\theta, \sigma_k)$ are uniquely determined 
by their moments (see \cite[Theorem 30.1]{billingsley}),  
convergence in moments to such a distribution implies that in distribution 
(see \cite[Theorem 30.2]{billingsley}). 


We use Martingale Central Limit Theorem (Theorem 35.12 in \cite{billingsley}) 
to prove that for any $v\in\cH$ the random variables $S_n := n^{-1/2}\ip{b(X_n),v}$ 
converge to a normal distribution $N(0,q(v))$ where $q(v)=\ip{S_{Pw}v,v}$. 
Consider the martingale array $S_{n,k}:=n^{-1/2}\ip{ b(X_k), v}$, $k=1,\ldots,n$, 
and put 
\[
Y_{n,k}:=S_{n,k}-S_{n,k-1}=n^{-1/2}\ip{\pi(X_{k-1})b(X_{k-1}^{-1}X_k),v}.
\] 
Since $X_{k-1}^{-1}X_k$ has the same distribution as $X_1$, 
one has 
\begin{align*}
\sum_{k=1}^n \ie{Y_{n,k}^2 \parallel X_1,\ldots,X_{k-1}}
 &= \frac{1}{n}\sum_{k=1}^n \ip{(\pi\otimes\pi)(X_{k-1})w,v\otimes v}\\
 &\stackrel{\IP}{\to} \ip{Pw,v\otimes v} = q(v),
\end{align*}
and, for every $\ve$,
\begin{align*}
\sum_{k=1}^n \ie{ Y_{n,k}^2 1_{\{ |Y_{n,k}| \geq\ve\}} }
 \le \ie{ \|b(X_1)\|^2\|v\|^2 1_{\{ \| b(X_1) \|\geq \ve n^{1/2} \}} }\to0.
\end{align*}
This shows that the array $S_{n,k}$ satisfies the assumption of 
the Martingale CLT    \cite[Thm 35.12]{billingsley}, and we can conclude that $S_{n,n}\Rightarrow N(0,q(v))$ in distribution. 

Now recall that $S_{Pw}=\sum_i\lambda_i E_i$ and $b=b_\infty+\sum b_i$, 
and take an orthonormal basis $\{ v_{i,j} : j=1,\ldots,\Tr(E_i)\}$ of $E_i\cH$. 
Then, by the previous paragraph, 
$n^{-1/2}\ip{b(X_n),v_{i,j}}$ converges in distribution to 
a centered Gaussian random variable $g_{i,j}$ with variance $q(v_{i,j})=\lambda_i$. 
Moreover, for any $\beta_{i,j}\in\IR$, the random variables 
$\sum_{i,j}\beta_{i,j}n^{-1/2}\ip{b(X_n),v_{i,j}}$
converge in moments to $N(0,q(\sum_{i,j}\beta_{i,j}v_{i,j}))$,
where 
\[
q(\sum_{i,j}\beta_{i,j}v_{i,j})
=\sum_{i,j} \beta_{i,j}^2\lambda_i = \sum_{i,j}\beta_{i,j}^2q(v_{i,j}).
\]
This means that the family $\{ \ip{n^{-1/2}b(X_n),v_{i,j}} \}_{i,j}$  
are asymptotically independent as $n\to\infty$.
Thus, for any $k\in\IN$, one has
\[
 \frac{1}{n}||\sum_{i=1}^k b_i(X_n)||^2 
= \sum_{i=1}^k \sum_j |n^{-1/2}\ip{b(X_n),v_{i,j}}|^2
\Rightarrow\sum_{i=1}^k\sum_j \lambda_i g_{i,j}^2,
\]
where $g_{i,j}$ are independent standard centered Gaussian random variables. 
Since 
\[
\lim_k \sup_n \ie{(\frac{1}{n}|| \sum_{i>k} b_i(X_n)||^2)^d}
 \le \lim_k C_d \| \sum_{i>k} b_i \|_{L^2(\mu)}^{2d}
 =0
\]
where $C_d$ is a constant independent of $k$ 
(by Lemma~\ref{lemma:aprioribound}), one has 
\[
\lim_n \ie{(\frac{1}{n}\| \sum_{i} b_i(X_n)\|^2)^d}
 = \lim_k\lim_n \ie{(\frac{1}{n}\| \sum_{i=1}^k b_i(X_n)\|^2)^d} 
\] 
for every $d$. 
Also, since $\frac{1}{n}\|b_\infty(X_n)\|^2\to \|b_\infty\|_{L^2(\mu)}^2$ in moments 
by the first half of the proof, 
one has 
$\frac{1}{n}\|b(X_n)\|^2 \to \|b_\infty\|_{L^2(\mu)}^2+\sum_{i,j} \lambda_ig_{i,j}^2 
\sim\chi(\theta,\sigma_k)^2$ in moments. 
\end{proof}

Recall that a group $G$ is said to have \emph{Shalom's property 
$H_\mathrm{FD}$} (\cite{shalom}) if every orthogonal representation $\pi$ with 
$\overline{H^1}(G,\pi)\neq0$ contains a non-zero finite-dimensional 
subrepresentation. In other words, $G$ has property $H_\mathrm{FD}$ 
if and only if every $\mu$-harmonic $1$-cocycle $b$ decomposes into 
a (possibly infinite) direct sum of finite-dimensional summands.
By Theorem~\ref{theorem:martingaleclt}, the latter happens for $b$ 
if and only if 
$\lim_n\mu^{*n}(\{ x \in G : \| b(x) \|\le c\sqrt{n} \})>0$ for all $c>0$. 

\begin{corollary}\label{cor:beta}
Assume either $(1)$ $\liminf_n\| \mu^{*n} - \mu^{*(1+\delta)n} \|_1<2$ for some $\delta>0$ or $(2)$ $\limsup_n\mu^{*n}(B_G(c\sqrt{n}))>0$ for all $c>0$. 
Then, $G$ has Shalom's property $H_{\mathrm{FD}}$. 
\end{corollary}
\begin{proof}
We prove a stronger statement that if $G$ does not have property $H_{\mathrm{FD}}$, then for every $\delta>0$ there are $c>0$ and 
a sequence $(E_n)_n$ of open subsets in $G$ such that 
$\mu^{*n}(E_n)\to1$ and $\mu^{*(1+\delta)n}(B_G(c\sqrt{n})E_nB_G(c\sqrt{n}))\to0$.

Suppose that there is $\mu$-harmonic $1$-cocycle 
$b\colon G\to\cH$ without a non-zero finite-dimensional summand. We can assume that this cocycle is normalized. 
Take any $0<\delta<1$.  
Put $c:=(20\| b \|_Q)^{-1}\delta$ and 
\[
E_n:=\{ x\in G : \|b(x)\|^2 < (1+\delta/4)n \}. 
\]
Then, for every $x\in E_n$ and $y,z\in B_G(c\sqrt{n})$ one has 
\[
\|b(yxz)\|^2 \le \|b(x)\|^2+2\|b(x)\|\|b(y)+\pi_{yx}b(z)\|+\|b(y)+\pi_{yx}b(z)\|^2
 < (1+\delta/2)n
\]
Hence the result follows from Theorem~\ref{thm:beta}.
\end{proof}

\begin{rem}
By Kingman's subadditive ergodic theorem, the linear rate of escape 
\[
\lim_n \frac{1}{n} |X_n(\omega)|_G = \lim_n \frac{1}{n}\IE |X_n|_G =: l_\mu
\]
exists and is constant for a.e.\ $\omega\in (G,\mu)^{\IN}$.
Hence either of the conditions (1) and (2) 
in Corollary~\ref{cor:beta} implies that $l_\mu=0$ 
and in particular that $G$ is amenable (\cite{guivarch}). 

\end{rem}

\begin{rem}
It is known that $\IZ\wr\IZ$ does not satisfy property $H_{FD}$ (\cite[5.4.1]{shalom}).
Shalom shows that any infinite amenable group with  $H_{\mathrm{FD}}$ admits a virtual quotient to 
$\IZ$
(\cite[4.3.1]{shalom}). 
By Corollary~\ref{cor:beta}, any non-degenerate random walk on  a group without virutal homomorphisms to 
$\IZ$ (or $\IZ\wr\IZ$)
does not satisfy either of the conditions (1) or (2). 
It is apparently on open problem whether the wreath product $\IZ^2\wr(\IZ/2\IZ)$ 
has property $H_{\mathrm{FD}}$ (see \cite[6.6]{shalom});
the simple  random walk 
on it does not satisfy either of the conditions (for "switch-walk-switch" random walks  it follows from  Dvoretzky--Erd\"os theorem (\cite{de,jp}) that the number of distinct sites of 
a simple random walk on $\IZ^2$ visited until the time $n$ is asymptotically equivalent to  $cn/\log(n)$, 
where $c>0$ is a constant.
\end{rem}

\section{More on the property $H_{FD}$} \label{sec:appli}
We elaborate on Corollary~\ref{cor:beta}. It says $G$ has property 
$H_\mathrm{FD}$ provided that $(G,\mu)$ satisfies the following property. 
We say a $\mu$-random walk $X_n$ is \emph{cautious} if 
\[
\limsup_n \IP( \max_{k=1,\ldots,n} |X_k|_G < c\sqrt{n} ) > 0
\]
for every $c>0$. We look at stability of this property under extension. 
Let $N$ be a closed normal subgroup of $G$ with a length $|\,\cdot\,|_N$ 
which may not be proper. 
We say $N$ is \emph{strictly exponentially distorted} in $G$ if 
there exists a constant $C \geq 1$ such that 
\[
\frac{1}{C}\log(|h|_N+1) - C \le |h|_G \le C\log(|h|_N+1) + C
\]
for all $h\in N$. We will denote by $|\,\cdot\,|_{G/N}$ the length 
induced by the compact generating neighborhood $QN$ of $e$ in $G/N$. 

\begin{proposition}
Let $N\triangleleft G$ be a closed normal subgroup which is strictly 
exponentially distorted, and let $\bar{\mu}$ be the push-out probability 
measure of $\mu$ to $G/N$. 
If $(G/N,\bar{\mu})$ is cautious, then so is $(G,\mu)$ and in particular 
$G$ has Shalom's property $H_\mathrm{FD}$. 
\end{proposition}
\begin{proof}
It suffices to show that there is a constant $D\geq1$ with the following 
property (cf.\ \cite[Lemma~3.4]{thompson}). 
Let $s_i\in G$ be such that $|s_i|_G\le 1$ and put 
$g_k:=s_1\cdots s_k \in G$ and $M_n:=\max_{k=1,\ldots,n} | g_k N|_{G/N} $. 
Then, one has $\max_{k=1,\ldots,n}|g_k|_G\le D( M_n + \log n + 1 )$. 
To show such $D$ exists, 
for each $k$, pick $h_k\in N$ such that $|g_k^{-1}h_k|_G = | g_k^{-1} N |_{G/N}\le M_k$. 
Then, $|h_{k-1}^{-1}h_{k}|_G\le 2M_k+1\le 3M_k$ and 
so $|h_{k-1}^{-1}h_{k}|_N\le \exp(4CM_k)$. 
Hence $|h_k|_N \le n\exp(4CM_n)$ for all $k\le n$ and so 
\[
\max_{k=1,\ldots,n}|h_k|_N\le C\log( 2n \exp(4CM_n) ) \le (D-1)( M_n + \log n + 1) 
\]
for some constant $D\geq 1$. Since $|g_k^{-1}h_k|_G\le M_n$, we are done. 
\end{proof}

Shalom (\cite[Theorem 1.13]{shalom}) has shown that polycyclic groups 
have property $H_\mathrm{FD}$ by invoking Delorme's theorem (\cite{delorme}) 
that connected solvable Lie groups have the corresponding property, and 
asked if there is another proof of $H_\mathrm{FD}$. 
It is plausible that all connected solvable groups are cautious. 
We note that in light of Osin's result (\cite{osin}) this problem reduces 
to the case for connected Lie groups with polynomial volume growth. 

\begin{corollary}
Let $K$ be a non-archimedean local field and $\IZ^d\curvearrowright K^n$ be 
a semi-simple linear action such that the semi-direct product $\IZ^d\ltimes K^n$ 
is compactly generated. Then, $\IZ^d\ltimes K^n$ has 
Shalom's property $H_\mathrm{FD}$. 
\end{corollary}
\begin{proof}
Let $\nu_0$ be the standard nearest neighborhood random walk on $\IZ^d$ and 
$\nu_1$ be a uniform probability measure on the compact subgroup 
$\{ x\in K : |x|\le 1\}$. 
Since $(\IZ^d,\nu_0)$ is cautious, for $\mu=\frac{1}{2}(\nu_0+\nu_1^{\otimes n})$, 
the random walk $(\IZ^d\ltimes K^n, \mu)$ is cautious. 
\end{proof}
\section{Harmonic cocycle $b_{\mu, U}$ constructed from differences of shifts of  $\mu^{*n}$}\label{sec:bmu}
In this section, we give a rather ``explicit'' (although we crucially use a non-principal ultrafilter) 
construction of a non-zero harmonic cocycle on a group that does not satisfy Kazhdan's property (T). 
In particular, when $G$ is a discrete finitely generated amenable group, a 
normalized $\mu$-harmonic cocycle $b_{\mu}$ will be obtained as an ultralimit of 
the sequence $\mu^{*n} - g\mu^{*n} \in \ell_2(G)$ after normalization. 
Throughout this section, we assume (in addition to Convention) that $\mu$ is compactly supported 
and $\mu=\mu'^{*2}$ for some symmetric probability measure $\mu'$ on $G$.

We fix a non-principal ultrafilter $\cU$ on $\IN$ and denote by ${\lim}_{\cU}$ 
the corresponding ultralimit. 
Then, the ultrapower Hilbert space $\cH^\cU$ of 
a given Hilbert space $\cH$ is defined to be 
\[
\cH^\cU := \ell_\infty(\IN;\cH)/\{ (v_n)_{n=1}^\infty : {\lim}_{\cU}\,\| v_n \|=0 \}
\]
with the inner product $\ip{ [v_n']_n, [v_n]_n}:= {\lim}_{\cU}\,\ip{v_n',v_n}$, 
where $[v_n]_n$ is the equivalence class of $(v_n)_n\in\ell_\infty(\IN;\cH)$. 
An orthogonal representation $\pi$ of $G$ on $\cH$ gives rise 
to the ultrapower representation $\pi^\cU$ on $\cH^\cU$ 
by $\pi^\cU_g[v_n]_n = [\pi_g v_n ]_n$. 
(NB: In general, the ultrapower representation is no longer continuous.)
We apply this construction to an orthogonal representation $(\pi,\cH)$ which 
admits an approximate invariant vectors but no non-zero invariant vectors. 
By definition, such an orthogonal representation exists if and only if 
$G$ does not satisfy Kazhdan's property (T) (see \cite{bhv}). 
\begin{lemma}\label{lem:non-t}
Let $(\pi,\cH)$ be an orthogonal representation 
which admits an approximate invariant vectors but no non-zero invariant vectors, 
and consider the positive and contractive operator $T:=\pi(\mu)$ on $\cH$. 
Then, there is a unit vector $v \in \cH$ such that the corresponding 
probability measure $\nu$ on $[0,1]$, defined by the formula 
\[
\int_0^1 t^n\,d\nu(t) = \ip{ T^n v,v},
\]
satisfies $1\in\supp\nu$ and $\nu(\{1\})=0$.
\end{lemma}
\begin{proof}
Let $E_T$ denote the spectral measure corresponding to the self-adjoint operator $T$.
Since $(\pi,\cH)$ admits approximate invariant vectors, 
the spectrum of $T$ contains $1$, which means that $E_T([1-1/n,1])\neq0$ for any $n$. 
Hence, there is a unit vector $v \in \cH$ such that $E_T([1-1/n,1])v \neq 0$ for any $n$. 
On the other hand, $E_T(\{1\})=0$ since $(\pi,\cH)$ has no non-zero invariant vectors. 
The probability measure $\nu(\,\cdot\,):=\ip{E_T(\,\cdot\,)v,v}$ corresponding to $v$ 
satisfies the desired conditions.
\end{proof}

Take $(\pi,\cH,v)$ as above and put $T = \pi(\mu)$. 
In case $G$ is a discrete finitely generated infinite amenable group, 
one can take $(\pi,\cH,v)$ to be $(\lambda,\ell_2(G),\delta_e)$ by 
Kesten's theorem (\cite{kesten}). 
Consider the coboundary $c_n\colon G\to \cH$ 
given by $c_n(g) = T^{n/2}v - \pi(g)T^{n/2}v$ and its normalization 
$b_n := \| c_n \|_{L^2(\mu)}^{-1}c_n$. 
We note that 
\[
\| c_n \|_{L^2(\mu)}^2 = 2\ip{ (T^{n} - T^{n+1})v,v} = 2\int_0^1 t^{n}(1-t)\,d\nu(t).
\]
We will define the cocycle $b_{\mu}$ to be the ultralimit of $b_n$. 
For continuity of $b_{\mu}$, we need equi-continuity of $b_n$'s. 
Observe that for every $g\in G$, one has 
\[
c_n(g) = - \int_G (\frac{d\mu}{dm} -g\frac{d\mu}{dm})(x) c_{n-2}(x)\,dm(x).
\]
Let $K=Q\supp \mu$ (recall that $Q$ is a relatively compact generating subset of $G$ 
and that $\supp\mu$ is assumed compact) and 
take a constant $C$ which satisfies $\| c \|_{K}\le C\| c \|_{L^2(\mu)}$ for 
every cocycle $c$ (see Lemma~\ref{lemmaequivnorms}).
Then by the above equality, for every $g\in Q$, one has 
\[
\| b_n(g) \| \le 
 \frac{\|c_{n-2}\|_{K}}{\|c_n\|_{L^2(\mu)}}\cdot
\| \frac{d\mu}{dm} -g\frac{d\mu}{dm} \|_{L^1}
\le C\frac{\|c_{n-2}\|_{L^2(\mu)}}{\|c_n\|_{L^2(\mu)}}\cdot
\| \frac{d\mu}{dm} -g\frac{d\mu}{dm} \|_{L^1}.
\]
Since $\frac{d\mu}{dm} \in L^1(G)$, 
the function $g\mapsto \| \frac{d\mu}{dm} -g\frac{d\mu}{dm} \|_{L^1}$ is continuous. 
Thus, equi-continuity of $b_n$'s follows from the following auxiliary lemma. 
\begin{lemma}\label{lem:auxiliary}
Let $\nu$ be a probability measure on $[0,1]$ such that $1\in\supp\nu$ and $\nu(\{1\})=0$. 
Then, $\gamma(n):=\int_0^1 t^n(1-t)\,d\nu(t)$ satisfies $\gamma(n)\searrow0$ 
and $\gamma(n+1)/\gamma(n) \nearrow 1$.
\end{lemma}
\begin{proof}
The first assertion is obvious. 
Since 
\[
\gamma(n+1) = \int t^{n/2}(1-t)^{1/2}\cdot n^{(n+2)/2}(1-t)^{1/2}\,d\nu(t)
\le\gamma(n)^{1/2}\gamma(n+2)^{1/2},
\]
the sequence $\gamma(n+1)/\gamma(n)$ is increasing and has 
a limit $\delta\le1$. Suppose for a contradiction that $\delta<1$. 
Then, one has $\gamma(n) \le C\delta^n$ and so 
$\int_0^1 t^n\,d\nu(t)=\sum_{k=n}^\infty\gamma(k)\le C'\delta^n$ 
for every $n$, where $C$ and $C'$ are some constant independent of $n$.
This implies $\supp\nu\subset[0,\delta]$, a contradiction.
Hence $\delta=1$.
\end{proof}

Since $b_n$'s are equi-continuous and $\| b_n(g) \| \le |g|_G\| b_n \|_Q$ is 
bounded for each $g$, the formula
\[
b_{\mu}(g) := [ b_n(g) ]_n \in \cH^\cU
\]
defines a continuous map such that 
$b_{\mu}(gh)=b_{\mu}(g)+\pi^{\cU}_g b_{\mu}(h)$. 
Since $b_{\mu}$ is continuous, the ultrapower orthogonal representation 
$\pi^{\cU}$ is continuous when restricted to $\cspan b(G)$.
Hence $b_{\mu}$ is a $1$-cocycle.
It is normalized:
\[
\| b_{\mu} \|_{L^2(\mu)}^2 = \int {\lim}_{\cU}\|b_n(x)\|^2\,d\mu(x)
= {\lim}_{\cU}\int \|b_n(x)\|^2\,d\mu(x)=1,
\] 
where, to interchange the ultralimit and integration, we have used the fact 
that $\mu$ is compactly supported and $b_n$'s are equi-continuous. 
The constructed $1$-cocycle $b_\mu$ may depend 
on the choice of a non-principle ultrafilter $\cU$ (see Theorem \ref{dependencesubsequence}), 
and we will write $b_{\mu,\cU}$ instead of $b_\mu$ when we want to emphasize the role 
of the ultrafilter $\cU$. 
The following reproves the results of Mok (\cite{mok}), Korevaar--Schoen (\cite{ks}), 
and Shalom (\cite{shalom-invent}) mentioned in Introduction. 

\begin{theorem}\label{thm:bmu}
Let $G$ be a compactly generated locally compact group 
which does not have Kazhdan's property {\normalfont (T)} and 
$\mu$, $(\pi,\cH,v)$, and $b_{\mu}$ be as above. 
Then, $b_{\mu}$ is a normalized $\mu$-harmonic cocycle. 
\end{theorem}

\begin{proof}
It only remains to prove that $b_{\mu}$ is harmonic. 
Put $\gamma(n)=\int t^{n}(1-t)\,d\nu(t)$. Then, one has
\[
\| \int b_n(x)\,d\mu(x) \|^2 
 = \frac{\gamma(n)-\gamma(n+1)}{2\gamma(n)}
 \to 0
\]
by Lemma~\ref{lem:auxiliary}.
Hence, for every $v'=[v_n']_n \in \cH^{\cU}$, one has
\begin{align*}
\ip{\int b_{\mu}(x)\,d\mu(x),v'} &= \int {\lim}_{\cU}\ip{ b_n(x),v_n'}\,d\mu(x)\\
 &= {\lim}_{\cU}\int \ip{ b_n(x),v_n'}\,d\mu(x) \\
 &=  {\lim}_{\cU} \ip{ \int b_n(x)\,d\mu(x), v_n'} = 0.
\end{align*}
This means $\int b_{\mu}(x)\,d\mu(x)=0$ and $b_\mu$ is harmonic. 
\end{proof}

In case $G$ is a discrete amenable group and $(\pi,\cH,v)=(\lambda,\ell_2(G),\delta_e)$, 
a computation yields that 
\[
\| c_n \|_{L^2(\mu)}^2 = 2 (\mu^{*n}(e) - \mu^{*n+1}(e))
\]
and 
\[
\|b_{\mu}(g)\|^2 = {\lim}_{\cU} \| b_n(g) \|^2 
 ={\lim}_{\cU}\frac{\mu^{*n}(e)- \mu^{*n}(g)}{\mu^{*n}(e) - \mu^{*n+1}(e)}.
\]
\begin{proof}[Proof of Theorem \ref{theoremcorollary}]
By Theorem~\ref{thm:beta} we know that 
$\ie{ | \frac{\|c(X_m)\|^2}{m} - 1 |^2}  \to0$ for any normalized harmonic 
cocycle $c$ without non-zero finite-dimensional summands.
We will show that in case $G$ does not admit any non-zero harmonic 
finite-dimensional cocycle (which is the case when $G$ is a finitely generated 
amenable group without virtually abelian infinite quotients), 
this convergence is uniform for normalized 
harmonic cocycles $c$ on $G$.
Indeed, we have seen in the proof of  Theorem~\ref{thm:beta}  that 
\[
\ie{ | \frac{\|c(X_m)\|^2}{m} - 1 |^2 }
 \le \frac{4}{m^2}\ip{  \sum_{k=1}^{m-1} (m-k)T^{k-1}w,w} + \frac{1}{m}\|c\|_Q^4\ie{|X_1|_G^4}
\to0
\]
for every normalized $\mu$-harmonic $1$-cocycle $c$, 
where $T = \int (\pi\otimes\bar{\pi})_g\,d\mu(g)$ 
and $w=\int (c\otimes\bar{c})(g)\,d\mu(g)$. 
Note that $\|c\|_Q$ is uniformly bounded by Lemma~\ref{lemmaequivnorms}.
Therefore, it suffices to prove that $\lim_k \|T^kw\|=0$ uniformly for $c$.
Suppose that the latter is not the case: there are $\ve>0$, 
a subsequence $k_m\to\infty$, and 
normalized harmonic cocycles $c_m$ with the corresponding $T_m$ and $w_m$ such that 
$\| T_m^{k_m}w_m\|\geq\ve$ for all $m$. 
Fix a non-principal ultrafilter $\cU$ and 
let $c_{\cU}$ denote the $\cU$-ultralimit cocycle of the sequence $(c_m)_m$,
with the corresponding objects denoted by $T_{\cU}$ and $w_{\cU}$.
Then, $c_{\cU}$ is a normalized harmonic cocycle. 
Moreover since $t^{2k}$ is decreasing in $k$ for any $t\in[-1,1]$, 
one has for each $k$ 
\[
\ip{T_{\cU}^{2k}w_{\cU},w_{\cU}}={\lim}_{\cU}\ip{T_m^{2k}w_m,w_m}
\geq {\lim}_{\cU}\ip{T_m^{2k_m}w_m,w_m} \geq\ve^2.
\]
Let $Q$ denote the spectral projection of $T_{\cU}$ corresponding to 
eigenvalues $\{ -1,+1\}$. 
Then, 
\[
\| Qw_{\cU}\|^2
=\lim_k \ip{k^{-1}(1+T^{2}+T^{4}+\cdots+T^{2(k-1)})w_U,w_U}\geq\ve^2. 
\]
Since $T_{\cU}^2 Qw_{\cU}=Qw_{\cU}$, 
the vector $Qw_{\cU}$ is invariant under $(\pi\otimes\bar{\pi})_g$ for all $g\in\supp\mu^{*2}$. 
However since $G_0:=\ip{\supp\mu^{*2}}$ has finite-index in $G$, 
it does not admit a non-zero $\mu^{*2}$-harmonic $1$-cocycle, 
which implies that $Qw_{\cU}=0$ (as discussed in the proof of Theorem~\ref{thm:beta}).
We have arrived at a contradiction. 

It follows that if $G$ satisfies the assumption of Theorem~\ref{theoremcorollary}, then 
$\ie{ | \frac{\|c(X_m)\|^2}{m} - 1 |^2}  \to0$ uniformly for 
normalized $\mu$-harmonic $1$-cocycles $c$.
In particular, 
\[
\lim_m\limsup_n \ie{ | \frac{\|b_n(X_m)\|^2}{m} - 1 |^2} 
 = \lim_m\sup_{\cU} \ie{ | \frac{\|b_{\mu,\cU}(X_m)\|^2}{m} - 1 |^2} = 0.
\]
(Note that $\limsup_n \lambda_n = \sup_U\lim_{\cU}\lambda_n$ for any bounded sequence $\lambda_n$.)
Since 
\[
\frac{1}{m}=\lim_n \frac{\mu^{*n}(e)-\mu^{*n+1}(e)}{\mu^{*n}(e)-\mu^{*n+m}(e)}
\] 
by 
Lemma~\ref{lem:auxiliary}, this completes the proof of Theorem~\ref{theoremcorollary} 
(after exchanging $\mu$ with $\mu^{*2}$).
\end{proof}

\section{$\ell_p$-thin subgroups} \label{sec:pthin}

\subsection{Definitions}

Take a finitely generated group $G$ equipped with a probability measure $\mu$, and ask again what information about its  subgroups and  quotient groups one can obtain by looking on  the behavior the random walk $(G,\mu)$. To ensure the existence of non-trivial quotients, we may search normal subgroups of $G$ defined by convolutions of $G$. A more general question one can ask is what are possible (not necessarily normal subgroups) defined in such terms.

\begin{definition} \label{definitionsubgroup} [$\ell_p$-thin subgroups $H_{\mu,p,q}$].
Let $G$ be an infinite group generated by a finite set $S$, and $\mu$ be a probability measure on $G$.
 Fix some $q\ge 0$, $p \ge 1$ and a sequence $n_i$ tending to $\infty$. Assume that $\mu$ is such that $(\mu^{*n})^q$ is in $l_p(G)$  for all $n$(this holds for example if
$\mu$ has finite support).
Let  $\alpha(n)$ denotes the maximum of  $\ell_p$ norm of  $(\mu^{*n})^q - g(\mu^{*n})^q$, where the maximum is taken over $g \in S$.
Consider  $g \in G$ for which 
   $||(\mu^{*n_i})^q - g(\mu^{*n_i})^q||_p/\alpha(n_i) \to 0$    as $i \to \infty$.  If $G$ contains at least two elements, then
by the triangular inequality in $\ell_p$, such elements form a subgroup of $G$, which we we call {\it the main $\ell_p$-thin subgroup} and which we 
denote  by  $H_{\mu,p,q, n_i}$
(and $H_{\mu,p}$ for short, if $n_i$ is specified and $q=1$).

Now we define $\ell_p$-thin subgroups associated an arbitrary function $\alpha(n)$. 
Consider $g$ such that 
   $||(\mu^{*n})^q - g(\mu^{*n})^q||_p/\alpha(n)$ tends to $0$    as $n$ tends to infinity.
By triangular inequality in $\ell_p$ such elements form a subgroup of $G$, which we denote 
$H_{\mu,p,q, \alpha}$. We call this subgroup $\ell_p$-thin subgroup associated to $\alpha(n)$.

\end{definition}

\begin{rem}\label{rem:0}
For $q=0$ in the definition above we use the convention $0^0=0$; the $\ell_1$ norm in this case is therefore
the cardinality of the symmetric differences of the supports of $\mu^{*n}$  and  $g \mu^{*n}$, that is the cardinality of the set of points $x$ such that either $x$ is in the support of $\mu^{*n}$ and $gx$ is not in this support or vice versa.  In the definition we have assumed that $p\ge 1$. We can extend the definition for the case $p=0$, defining
$\alpha(n)$ as the maximum of the cardinality of the support  of $(\mu^{*n})^q - g(\mu^{*n})^q$, where the maximum is taken over $g \in S$.
In this case we obtain $H_{0,1, \mu}= H_{1,0,\mu}$ for all $\mu$.
Observe that that if the support of $\mu$ is a finite symmetric generating set containing the identity, then the support of $\mu^{*n}$ is the ball of radius $n$ in the word metric associated to $S$.
\end{rem}

It is clear that the scaling sequence $\alpha(n)$ depends of a finite generating set $S$ up to multiplication by a constant only, and thus the definition of main $\ell_p$-thin subgroups does not depend on the choice of $S$.

In many situation the limit behavior of  $(\mu^{*n})^q - g(\mu^{*n})^q$ does not depend on the subsequence of possible $n$'s. However, in some situation this quantity, and the corresponding $\ell_p$-thin
subgroups may depend on the choice of a subsequence, see Theorem C and Corollary \ref{cor:lpthin}.

\begin{rem}\label{rem:changep}
If  $p\ge 1$,  it is known that a normalized sequence $v_n \in \ell_1(G)$ is almost invariant in $\ell_1$ with respect to the shift by some element  $g \in G$  if and only if $v_n^{1/p}$ (which is clearly a sequence in $\ell_p(G)$) is almost invariant in $\ell_p$ with respect to the shift by $g$
(see e.g. the proof of Theorem 8.3.2  in \cite{reiterbook}. This implies that the main $\ell_p$-thin subgroups satisfy
$H_{p,1} = H_{1,p}=H_{p/q, q}$ for any $p,q\ge 1$ whenever $(\mu^{*n_i})^p$ does not admit a subsequence of almost invariant vectors in $\ell_1$. This happens for example for $p=2$,  if $G$ non-amenable and for $p=1$ if the Poisson boundary of $(G,\mu)$ is non-trivial, for all $n_i$ (\cite{kaimanovichvershik}).
\end{rem}

It is possible that the statement of Remark \ref{rem:changep} remains valid without the assumption of non-almost-invariance.

Instead $(\mu^{*n})^q$ in the Defintion \ref{definitionsubgroup}, one can consider
more generally  a sequence of  
functions $f_i$ and consider 
the difference of corresponding shifted functions, as a function of $g$.

We have already remarked  that for $p=2$, $q=1$, $\mu$ being equidistributed on a finite symmetric set of $G$,  the values of $b_{\mu,U}$ are defined by the unmarked Cayley graph of $G$. In particular, 
 for $p=2$, $q=1$ and $\mu$ being a measure equidistributed on a finite generating set $S$ , the $\ell_p$-thin subgroups can be described in terms of unmarked Cayley graph of $(G,S)$:

\begin{remark} \label{lemmaprobareturn}
$p=2$, $q=1$,  $\mu$ is symmetric measure on $G$. Fix a sequence $\alpha_i$, tending to infinity.
An element $g$ belongs to the subgroup $H_{\mu, 2, 1,  \alpha }$ if and only if
$$
( \mu^{2n}(e) - \mu^{*2n}(g))/\alpha_n^2) \to 0
$$
as $n \to 0$.
In particular, if $\mu$ is equidistributed on a finite symmetric generating set $S$, subgroups
 $H_{\mu, 2, 1, \alpha_i}$ are defined by unmarked Cayley graph of $(G,S)$.

\end{remark}
{\bf Proof.} Observe that $
| g \mu^{*n} -\mu^{*n}|_2^2=  | \mu^{*n}|_2^2 + |g \mu^{*n}|_2^2  - 2 \ip{  \mu^{*n}, g \mu^{*n}} =
2 | \mu^{*n}|_2^2  - 2 \ip{ \mu^{*n}, g \mu^{*n}} =2( \sum_{x\in G} (\mu^{*n}(x))^2 -  
\sum_{x\in G} \mu^{*n}(x)\mu^{*n}(gx))$,
Since $\mu$ is symmetric, this is equal to 
$2( \sum_{x\in G} \mu^{*n}(x) \mu^{*n}(x^{-1}) -  
\sum_{x\in G} \mu^{*n}(gx)    \mu^{*n}(x^{-1}   ) 
=
2( \mu^{2n}(e) - \mu^{*2n}(g))$.
If $\mu$ is equidistributed on a finite symmetric generating set $S$, observe that $\mu^{*2n}(e)$ and $\mu^{*2n}(g)$ are defined by the unmarked Cayley graph of $(G,S)$ and the vertex in this Cayley graph corresponding to $g$.

\begin{rem}
In a particular case when $q=1$, $p=2$ and $G$ is non-amenable, the  main $\ell_2$-thin  subgroup in \ref{definitionsubgroup} coincides with the group, studied by Elder and Rogers in \cite{elderrogers}. However, if $q=1$, $p=2$ and $G$ is amenable, the group
defined in the above cited paper coincides with $G$, while the main $\ell_2$-thin subgroup $H_{\mu,p}$ is never equal to $G$ (for any infinite group $G$). 
\end{rem}

Now assume that $\mu$ has finite support, and consider the mappings $b^{p,q}_{\mu,U}$, defined in the introduction.
Namely, for any non-principal ultrafilter $\cU$ on $\IN$,
put $\alpha^{p,q}(n) = \max_{s\in S} \| (\mu^{*n})^q - s (\mu^{*n})^q \|_p$ 
and 
define the cocycle $b^{p,q}_{\mu,U}\colon G\to\ell_p(G)^U$ by 
\[
b^{p,q}_{\mu,U}(g)=[\alpha^{p,q}(n)^{-1}((\mu^{*n})^q - g (\mu^{*n})^q)]_n \in \ell_p(G)^U.
\]
The cocycle $b^{p,q}_{\mu,U}$ is independent, modulo scalar multiple,
of the choice of the finite generating subset $S$. 
We note that $\ell_p(G)^U$ is an abstract $L_p$-space 
on which $G$ acts isometrically.
Hence $b^{p,q}_{\mu,U}(G)$ is contained in a $G$-invariant 
separable $L_p$-subspace of $\ell_p(G)^U$.

\begin{lemma}  \label{directproducts} 
1)  [Direct products, $q=0$, $p\ge1$]  Let $G$ be
a direct product of $A$ of subexponential growth and $B$ of exponential growth, and let 
$\mu=\mu_A \times \mu_B$ where $\mu_A(e)>0$. Then there exists a subsequence $n_i$ such that 
 subgroup $H_{\mu,0,p}(G) = H_{\mu,p,0}(G) $ contains $A$. Moreover, for any  $n_i$ as above, any ultrafilter $U$ such that $U(n_i)=1$
 for $q=0$ and $p\ge 1$ 
satisfy
$b^{p,q}_{U,\mu} = b^{p,q}_{U,\mu_B}$.

2) [Direct product, $p=q=1$] 
 Let $G$ be a direct product of a group $A$ and $B$;  let $\mu_A$, $\mu_B$ be non-degenerate measures on $A$ and $B$ such that the
  Poisson boundary of a random walk $(A, \mu_A)$ is trivial 
 and Poisson boundary of  $(B,\mu_B)$ is non-trivial. Put
$\mu=\mu_A \times \mu_B$. Then for any choice of $n_i$ the main $\ell_1$-thin subgroup $H_{\mu,1,1} (G)$ contains $A$. Moreover, for any ultrafilter $U$ it holds
$b^{1,1}_{U,\mu} = b^{1,1}_{U,\mu_B}$.

3)  [Direct products, $q=1$, $p=2$]  Let $G$ be  a direct product of an amenable $A$ and non-amenable group $B$, $\mu = \mu_A \times \mu_B$. Then
for any $n_i$, the main $\ell_2$-thin subgroup $H_{\mu,1,2}(G)= H_{\mu,2,1}(G) $ contains $A$.
Moreover, for any ultrafilter $U$ it holds
$b^{2,1}_{U,\mu} = b^{2,1}_{U,\mu_B}$.

\end{lemma}

{\bf Proof.}
First we prove the claims of 1), 2), 3) about $\ell_p$-thin subgroups.
Observe that  since $B$  is of exponential growth, for any finite set $S$ there exists $v>1$ such that $v_{G,S}(n) \ge v^n$ for all $n$. This implies that
 for each  finite generating set $S_B$ of $B$ and each $C_1 <1$ there exists  $C_2>0$ such that for all $n$
at least $C_1 n$ among balls of radius $i=1, ..., n$ have boundary greater  than $C_2  v_{B,S_B}(i)$.
(Indeed, otherwise $v_{B,S_B}(n) \le R_B^{n(1-C_1)}(1+C_2)^{C_1 n}$, where $R_B$ denotes the cardinality of $B,S_B$,
and taking $C_1$ close to $1$ and $C_2$ close to $0$ we would get a contradiction).

Since  $A$ is of subexponential growth, for each $C$ and any $\epsilon_1, \epsilon_2>0$
at least $(1-\epsilon_2) n$ among the the balls of radius $i=1, ..., n$ have boundary
at most $\epsilon_2  v_A(i)$.

Consider a generating set
 $S=S_A \times S_B$, 
where $S_A$, $S_B$ are generating sets of $A$, $B$ respectively.
We have  $B_S(i)= B_{S_A} (i) \times B_{S_B}(i)$. 
Here $B_{G,S}(i)$ denotes the ball of radius 
$i$ in $G,S$.
Observe also that for $S \in S_A$ it holds 
$s B_{G,S}(i) \setminus B_{G,S}(i) =  s (B_{A,S_A}(i) \times B_{B,S_B}(i)) 
 \setminus (B_{A,S_A}(i) \times B_{B,S_B}(i) )=( s B_{A,S_A}(i)\setminus B_{A,S_A}(i)) \times B_{B,S_B}(i)$,
and the cardinality of this set is at most  $2 (v_{A,S_A}(i) - v_{A,S_A}(i-1)) v_{B,S_B}(i)$, and with the same argument
the cardinality of  $s B_{G,S}(i) \setminus B_{G,S}(i)$, for $s \in S_B$ is at least
$ 2/|S_B| (v_{B,S_B}(i) - v_{B,S_B}(i-1) v_{A,S_A}(i)$ for some $s \in S_B$.
This shows that there exists a sequence $n_i$, tending to infinity,  such that 
$$
\frac{v_{A,S_A}(n_i) - v_{A,S_A}(n_i-1)}{v_{A,S_A}(n_i)}  
\frac{v_{B,S_B}(n_i)}{v_{B,S_B}(n_i) - v_{B,S_B}(n_i-1) } 
$$
tends to $0$ 
as $i$ tends to infinity. 
By Remark  \ref{rem:0}
we know that  for any group
 it holds $H_{\mu,0,p}(G) = H_{\mu,p,0}(G)$.
 Note that for any $n_i$ as above the  this 
thin subgroup $H_{\mu,0,p}(G) = H_{\mu,p,0}(G)$
with respect to a subsequence $n_i$ contains all $s \in S_A$. Therefore, in this case this subgroup contains $A$.

2)  We recall that  $\mu^{*n}= \mu_A^{*n}\mu_B^{*n}$.
Take $a \in A$. Observe that
$a\mu^{*n} -\mu^{*n} = (\mu_A^{*n}- a \mu_A^{*n}) \mu^{*n}(B)$.
It holds therefore $  ||  (\mu^{*n}- a \mu^{*n})  ||_1   =   ||(\mu_A^{*n}- a \mu_A^{*n})||_1$.
Since the non-degenerate walk $(A,\mu_A)$ has trivial Poisson-Furstenberg boundary, for any $a \in A$ it holds
$  ||  (\mu_A^{*n}- a \mu_A^{*n})  ||_1  \to 0$ as $n$ tends to $\infty$, and therefore
$  ||  (\mu^{*n}- a \mu^{*n})  ||_1  \to 0$ as $n$ tends to $\infty$ (see Kaimanovich Vershik \cite{kaimanovichvershik}). The above mentioned characterization also shows that since the Poisson boundary of $(B,\mu_B)$ is non-trivial, there exists $b\in B $ such that
$  ||  (\mu_B^{*n}- b \mu_B^{*n})  ||_1 \ge c>0$, and hence $||  (\mu^{*n}- b \mu^{*n})  ||_1 \ge c>0$ for some positive constant $c$ and all $n$.

3)  For $g=(g_1, g_2)$, $g_1 \in A$, $g_2 \in B$, 
$\mu^{*n}(g_1, g_2) = \mu_A^{*n}(g_1) \mu_B^{*n}(g_2)$. For $h \in A$, $\mu^{*n}(h(g_1, g_2)) / \mu^{*n}(g_1,g_2)=
\mu_A^{*n}(hg_1)/\mu_A^{*n}(g_1) \to 1$ as $n \to  \infty$, by \cite{avez} since $A$ is amenable \cite{avez}.
Analogously,   for $h\in B$ it holds $\mu^{*n}(h(g_1, g_2)) / \mu^{*n}(g_1,g_2)=\mu_B^{*n}(hg_2)/ \mu_B^{*n}(g_2)
 \to C_h$, where $C_h \ne 1$ for some $h$ among generators of $B$, since $B$ is non-amenable \cite{avez}.
This implies that the scaling sequence $\alpha(n)$ is equivalent up to multiplicative constant to  $\mu^{*n}(e)= \mu_A^{*n}(e) \mu^{*n}_B(e)$. 
Using Remark \ref{lemmaprobareturn} we conclude that  for all $s \in S_A$ $||s \mu^{*n} -\mu^{*n}||_{2}/\alpha(n) \to 0$, and hence any $s \in A$, $s$ belongs to the $\ell_2$ thin subgroup for $q=1$, $p=2$.
By Remark \ref{rem:changep} we know that 
under assumption of 3) it holds $H_{\mu,1,2}(G) = H_{\mu,2,1}(G)$. 

Now to prove the claims about the cocycles, take  $g= (a, b ) \in A \times B$, put $g'=(e,b)$ and $g''=(a,e)$. It holds $g= g'g''$. Under the assumption on $p$ and $q$ in 1), 2), 3) observe  that 
$$
           || (\mu^{* n_i})^q - 
g' (\mu^{* n_i})^q||_p - || (\mu^{* n_i})^q - 
g'' (\mu^{* n_i})^q||_p     \le  || (\mu^{* n_i})^q - 
g (\mu^{* n_i})^q||_p  \le
$$
$$
\le  || (\mu^{* n_i})^q - 
g' (\mu^{* n_i})^q||_p + || (\mu^{* n_i})^q - 
g'' (\mu^{* n_i})^q||_p
$$
and that $ || (\mu^{* n_i})^q - 
g' (\mu^{* n_i})^q||_p = || (\mu_B^{* n_i})^q - 
g' (\mu_A^{* n_i})^q||_p  || \mu_A^{n_i}  ||_p$.
This allows us to use 1) of Remark \ref{rem:cocyclesproduct} and completes the proof of the Lemma.

\begin{rem}\label{rem:cocyclesproduct}

$G=A \times B$, $\mu=\mu_A \times \mu_B$, $S= S_A \times S_B$, $S_A$ and $S_B$ are finite generating sets of $A$ and $B$.

1) Let $\alpha^{p,q}_G(n)$ be the maximal $\ell_p$ norm of $(\mu^{*n})^q - s(\mu^{*n})^q$, where the maximum is taken over $s\in S$; and 
let $\alpha^{p,q}_A(n)$ be the maximal $\ell_p$ norm of $(\mu_A^{*n})^q - s(\mu_A^{*n})^q$, the maximum is over $s\in S_A$ and   $\alpha^{p,q}_B(n)$ is defined analogously.
Let $\tilde{\theta}_\mu(n)$ be equal to $\alpha^{p,q}_G(n)$ divided by the $\ell_p$ norm
of $(\mu^{*n})^q$.
If $\tilde{\theta}^{p,q}_A(n_i) /\tilde{\theta}^{p,q}_B(n_i)$ tends to zero for some sequence $n_i$ and $U$ is 
a non-principal ultrafilter such that $U(\{n_i \})=1$, then
$b^{p,q}_{\mu,U} = b^{p,q}_{\mu_B,U}$.

2) Take $q=1$, $p=2$. 
Put $\theta(n):=  (\mu^{* 2n} - \mu^{*2n+1})/\mu^{*2n}$. Then
$ \theta(n)= \tilde{\theta}_\mu(n)^2$. In particular, if
$\theta_A(n_i)/ \theta_B(n_i)$ tends to zero and $U$ is a non-principal ultrafilter such that $U(\{n_i\})=1$,
then the corresponding harmonic cocycle is defined by that of $B$, that is
$b_{\mu,U} = b_{\mu_B,U}$.

\end{rem}

\begin{rem} The fact that $A \times B$, $A$ is of subexponetial growth, $B$ is of exponential growth,  satisfy the claim of 1), Lemma \ref{directproducts}  not only for some sequence $n_i$ but for all sequences
can be shown to be  equivalent to a positive answer to both following questions A): is it true that no subset of balls is a Foelner sequence in $A$? B): Is it true that all balls form a Foelner sequence in $A$?

To our knowledge, it is not known whether to answer to  A) is positive for all groups of exponential growth (this question is mentioned e.g. in \cite{tesseravolume}), 
 and whether the answer to  B) is positive for all groups of subexponential growth.
\end{rem}

\begin{example}[Dependance of $\ell_p$-thin subgroups on $p$] \label{dependancep}

Let $G= F_m \times \Z^d \wr A$, where $m \ge 2$, $d\ge 3$ and $A$ is a finite group containing at least two elements.
Let $\mu$ be a  non-degenerate symmetric finitely supported measure.
Then $\ell_2$-thin subgroup is not equal to $\ell_1$-thin subgroup.
\end{example}

{\bf Proof.}  Observe that the   $\ell_2$-thin subgroup $H_{\mu, 1, 2} =H_{\mu, 2, 1}$  contains $\Z^d \wr A$  by 3) of Lemma \ref{directproducts} 
(in fact, it is equal to  $\Z^d \wr A$), while there exists $g \in \Z^d \wr A$ which does not belong to 
$\ell_1$-thin subgroup  since the Poisson boundary of $\Z^d \wr A$ is non-trivial.

\begin{remark}\label{rem:notnormal}
Let $G = C \wr A$, where $C$ is an infinite group of at least cubic growth  and $A$ is a finite group containing at least two elements.
Let $\mu$ be a symmetric finitely supported  "switch-walk-switch" measure on  $G$.
One can show that 
 $H_{\mu,1,1}$ is a finite subgroup of $G$.  One can also show that 
 for any integer $k\ge 0$ there exists $\mu$ as above such that
$H_{\mu,1,1}$ is isomorphic to $A^m$. In particular, this main $\ell_1$ thin subgroup $H_{\mu,1,1}$ depends on the choice of a finitely supported symmetric measure $\mu$ and this subgroup is not normal.
\end{remark}

{\bf Proof of Theorem \ref{dependencesubsequence}.}  Assume $d=2$ (the general case $d\ge 2$ is analogous).

We construct $G_1$ and $G_2$  as piecewise-automatic groups with returns  of automata $\tau_1$, $\tau_2$, where $\tau_1, \tau_2 : A\times X \to A$, the group generated by $(A,\tau_1)$ is of intermediate growth, $\tau_2: A \times X \to A$,  the group $H_2$ generated by $(A,\tau_2)$ is non-amenable,  and the action of $A$, considered as generators of $H$, is contracting for the action of $\tau_1$ for each brach of the rooted tree (see \cite{erschlerpiecewise}).

More precisely, we chose automata $\tau_1$ and $\tau_2$ with the following properties: $\tau_2$ is  a finite state automaton, containing $e,a,b, c, d$ as its states, such that $e$ acts trivially and $a,b,c,d$ generate the free product $\mathbb{Z}/2\mathbb{Z} * (\mathbb{Z}/2\mathbb{Z}   + \mathbb{Z}/2\mathbb{Z})$ in the group generated by $\tau_2$.

 If the  states of $\tau_2$ are $e, a, b, c, d$ and the alphabet is $0$, $1$, we take 
as $\tau_1$ the standard finite state automaton  for the first Gigorchuk group ($A=\{e,a,b,c,d \}$, $X=\{0, 1\}$.
In this case we can  take as $G_1$ and $G_2$ either piecewise-automatic group or a piece-wise automatic group with returns defined by $\tau_1$, $\tau_2$ and $t_i, T_i$, $i\ge 1$, $T_{i-1} < t_i < T_{i}$.
We do not know if $\tau_2$ as above exists, and therefore we consider as in \cite{erschlerpiecewise}
an automaton $\tau_2$ with the space of states  possibly larger than $e, a, b,c,d$ (such automata exist by the result of Olijnyk \cite{olijnyk}, that shows that any free product of finite groups imbeds in a group generated by a finite state automaton), and
we take as $\tau_1$ the standard finite state automaton for the first Grigorchuk group,
(extended to some larger alphabet than $0$ and $1$ if the alpaheth of $\tau_2$ contains more than two letters) 
and consider the corresponding piecewise automatic group with returns $G_{\tau_1, \tau_2}(t_i, T_i)$.

To construct $G_1$ and $G_2$, we fix $\tau_1$, $\tau_2$ and construct sequences $t^1_i, T^1_i$  
and 
$t^2_i, T^2_i$  ($T^1_{i-1} < t^1_i < T^1_{i}$,  $T^2_{i-1} < t^2_i < T^2_{i}$)
by a simultaneous inductive procedure and we put 
$G_1=G_{\tau_1, \tau_2}(t_i^1, T^1_i)$ and $G_2=G_{\tau_1, \tau_2}(t_i^2, T^2_i)$ .
 
We need the following  properties of  piece-wise autmatic group with returns $G_{\tau_1, \tau_2}(t_i, T_i)$
(see proof of Proposition $1$ in \cite{erschlerpiecewise}). There exist $\Psi: \mathbb{N} \to \mathbb{N}$ and for each $i$ there exist "comparison groups" $\mathcal{A}(t_1, T_1, t_2, T_2, \dots t_i)$ and $\mathcal{B}(t_1, T_1, t_2, T_2, \dots t_i, T_i)$, such that the following holds for all non-decreasing sequences $t_i$, $T_i$: 

\begin{enumerate}

\item all groups  $\mathcal{A}(t_1, T_1, t_2, T_2, \dots T_{i-1})$ have a finite index subgroup which imbeds as a subgroup in a finite direct power of the the first Grigorchuk group $G_1$ (generated by $(A, \tau_1)$,

\item all groups  $\mathcal{B}(t_1, T_1, t_2, T_2, \dots T_{i-1}, t_i)$ have a finite index subgroup which admits a surjective homomorphism  to the group, generated by the automaton $(A,\tau_2)$,

\item 
the balls of radius $\Psi(t_i)$ in 
  $G(t_1, t_2, ... , T_1, T_2, ...)$  and $\mathcal{A}(t_1, T_1, t_2, T_2, \dots t_{i-1}, T_{i-1})$ coincide,

\item the balls of radius $\Psi(T_{i})$ in 
  $G(t_1, t_2, ... ,T_1, T_2, ...)$  and $\mathcal{B}(t_1, T_1, t_2, T_2, \dots T_{i-1},  t_i,)$ coincide.

\end{enumerate}

Let $G,S_G$, $H, S_H$ be finitely generated groups such that the balls of radius $R+C$  in the marked Cayley graphs of $G,S_G$, $H, S_H$ coincide. Let $\mu_H$ and $\mu_G$ are measures which are equal after the identifications of these balls and such that $ l_G(s) \le C$ for any $s$ in the support of $\mu_G$. Observe that for any $n \le R$ the scaling functions in the definition of $\ell_p$-thin subgroups are equal : $\alpha_{G,\mu_G, p}(n)= 
\alpha_{H,\mu_H, p}(n)$,  $\alpha'_{G,\mu_G, p}(n)= 
\alpha'_{H,\mu_H, p}(n)$ and for each $g$ in the ball of radius $C$ in the Cayley graph of $(G,S_G)$ $\ell_p$ norms of $g (\mu_G^{*n})^q - (\mu_G^{*n})^q$ 
are equal to the $\ell_p$ norm of $h(\mu_H^{*n})^q - (\mu_H^{*n})^q$ for $h$ being the corresponding element in the ball of radius $C$ of $(H, S_H)$.

Suppose that we have chosen already  $t^1_1, T^1_1, t^1_2, T^1_2, \dots T^1_{i-1}$ and 
$t^2_1, T^2_1, t^1_2, T^2_2, \dots t^2_i$.  For any $\epsilon>0$ there exist $M_i$ such that for all $M_i'>M_i$ there exists $M_i^{*}$  with 
the following property. For any $t^1_i>M_i^{*}$ and $T^2_{i}>M_i^{*}$ and any $n: M_i < n <M_i'$
the ratio of $\ell_p$ norms $s_1 (\mu^{*n})^q -  (\mu^{*n})^q$  and $s_2 (\mu^{*n})^q -  (\mu^{*n})^q$ in $G=G_1\times G_2$
is smaller than $\epsilon$ for all $s \in S_1$ and some $s\in S_2$.

To prove this , we combine the  observation about Cayley graphs above with  the claims 1), 2), 3) of Lemma \ref{directproducts}, 
for
$\mathcal{A}=\mathcal{A}(t^1_1, T^1_1, t^1_2, T^1_2, \dots T^1_{i-1})$, $\mathcal{B}= \mathcal{B}(t^2_1, T^2_1, t^1_2, T^2_2, \dots t^2_i)$. Tthe group $A$ is of intermediate growth and hence  this group is  amenable and finitely supported random walks have trivial boundary, $B$ has a finite index subgroup sujecting to a non-amenable group, and hence non-amenable.

Now suppose that we have chosen already $t^1_1, T^1_1, t^1_2, T^1_2, \dots t^1_i$ and 
$t^2_1, T^2_1, t^1_2, T^2_2, \dots t^2_i, T^2_i$.
For any $\epsilon>0$ there exist $N_i$  such that for all $N_i'>N_i$ there exists $N_i^{*}$  with 
the following property.
For any $T^1_{i} > N_i^{*}$ and $t^2_{i+1}>M_i^{*}$ and any $n: N_i < n <N_i'$
the ratio of $\ell_p$ norms of $s_2 (\mu^{*n})^q -  (\mu^{*n})^q$  and $s_1 (\mu^{*n})^q -  (\mu^{*n})^q$ in $G=G_1\times G_2$
is smaller than $\epsilon$ for all $s_2 \in S_2$ and some $s_1\in S_1$.

This implies that  for some choice of $t^1_i$, $T^1_i$ and  $t^2_i$, $T^2_i$
there exist  sequences $n_i, m_i$ tending to infinity, such that the following holds.
 The ratio of $\ell_p$ norms of $s_1 (\mu^{*n_i})^q -  (\mu^{*n_i})^q$ and the scaling sequence $\alpha(n_i)$ tend to $0$ for all $s_1\in S_1$. This implies that  all $s_1\in S_1$, as well as all $g\in G_1$ belong to the main $\ell_p$ thin subgroup  $H_{\mu, p,q}$ ,  corresponding to $n_i$.
The ratio of $\ell_p$ norms of $s_2 (\mu^{*m_i})^q -  (\mu^{*m_i})^q$ and the scaling sequence $\alpha(m_i)$ tend to $0$ for all $s_2 \in S_2$. This implies that  all $s_2\in S_e$, as well as all $g\in G_2$ belong to the main $\ell_p$ thin subgroup  $H_{\mu, p,q}$ , corresponding to $m_i$.  Consider an ultrafilter $U_m$ such that $U(m_i)=1$ and an ultrafilter $U_n$ such that $U(n_i)=1$.
Using 1), 2), 3) of Lemma \ref{directproducts} we also observe that $b^{p,q}_{\mu,U_n}$ is equal to
$b_{\mu_2,U_n}^{p,q}$  and that  $b_{\mu,U_m}^{p,q}$ is equal to
$b^{p,q}_{\mu_1,U_m}$.

\begin{corollary}\label{cor:lpthin}
Let $G_i$, $\mu_i$ be as in the formulation of Theorem \ref{dependencesubsequence}.
Take $q=0$, $1$ or $2$ and $p=1$ or $2$. For each $j: 1 \le j \le D$
there exists $n_{i,j}$ such that for all the main $\ell_p$-thin subgroup $H_{\mu, p, q} $of $G$ with respect to $n_{i}= n_{i,j}$ contains
$\prod _{k: k \ne j} G_k$. In particular, there exist at least $D$ not equal $\ell_p$-thin subgroups. 
\end{corollary}

\end{document}